\documentclass[preprint,11pt]{elsarticle}

\usepackage{amssymb}
\usepackage{amsmath}
\usepackage{amsthm}
\usepackage{bm}
\usepackage[ruled, linesnumbered]{algorithm2e}
\usepackage{multirow}
\usepackage{booktabs}
\usepackage[colorlinks=true,linkcolor=red,citecolor=blue,urlcolor=blue]{hyperref}
\usepackage{fullpage}
\usepackage{float}

\newcommand{\bx}{\bm{x}}
\newcommand{\ba}{\bm{a}}
\newcommand{\bw}{\bm{w}}
\newtheorem{thm}{Theorem}
\newtheorem{df}[thm]{Definition}
\newtheorem{proposition}[thm]{Proposition}
\newtheorem{assumption}[thm]{Assumption}
\newtheorem{lma}[thm]{Lemma}
\newtheorem{corollary}[thm]{Corollary}
\newdefinition{rmk}{Remark}
\journal{Computer Methods in Applied Mechanics and Engineering}

\begin{document}

\begin{frontmatter}
\title{An Adaptive Transferable Neural Network for Partial Differential Equations with Anisotropic or Sharp Solutions}

\author[a]{Jiacong Qu}
\ead{jiacongqu@buaa.edu.cn}
\author[a,b]{Liyong Zhu\corref{cor1}}
\ead{liyongzhu@buaa.edu.cn}

\address[a]{School of Mathematical Sciences, Beihang University, Beijing 100191, China}
\address[b]{LMIB(Beihang University), Ministry of Education, Beijing, 100191, China}
\cortext[cor1]{Corresponding author.}

\begin{abstract}
Transferable neural network (TransNet) has demonstrated great potential for solving partial differential equations (PDEs) with high accuracy and efficiency. However, it struggles significantly when dealing with PDEs featuring anisotropy or sharp solutions. In this work, we propose an adaptive transferable neural network (A-TransNet), a novel machine learning framework that adaptively expands neurons in regions with large solution gradients to effectively capture complex solution behaviors. Instead of relying on training or random sampling, the location parameters of newly added neurons are determined by integrating the active subspaces method and local gradient analysis based on posterior gradient information. This adaptive strategy enables neurons to be allocated where they are most needed, thereby improving the resolution of anisotropic and sharp features without  additional computational costs. Moreover, the shape parameters of new neurons are automatically determined from the neuron density induced by their locations, further enhancing the network representation capability. An iterative neuron expansion strategy is developed, together with rigorous analysis of stability, a posteriori error estimation, and convergence of the proposed A-TransNet method. Extensive numerical experiments on challenging PDE problems with anisotropic and sharp solutions demonstrate that A-TransNet achieves superior accuracy, efficiency, and robustness compared with existing methods.
\end{abstract}

\end{frontmatter}

\section{Introduction}
A large class of practically relevant PDEs arises from anisotropic media, multiscale transport and phase-transition phenomena, whose solutions are inherently anisotropic \cite{Leung2013NonlinearPDE} and often develop sharp layers \cite{Debnath2005NPDE}, boundary layers \cite{Achdou2014PDEMacro}, shocks or steep transition fronts \cite{Gomez2019PhaseReview} that pose severe numerical challenges. Traditional numerical methods, including finite difference methods (FDM) \cite{LeVeque2007FDM}, finite element methods (FEM) \cite{Thomee2007FEM}, finite volume methods (FVM) \cite{Moukalled2016FVM}, and spectral methods, perform satisfactorily for problems with smooth solutions, yet they struggle significantly when dealing with PDEs featuring anisotropy or sharp solutions. Against this background, Adaptive Finite Element Methods (AFEMs) \cite{Kopteva2015AnisoAFEM, Cier2021AdaptiveFEM}, Rotated Staggered-Grid Finite-Difference (RSFD) \cite{Xie2024RotatedFD}, and anisotropic variational mesh adaptation
\cite{Rahmani2025AnisoMesh} have been proposed. While these approaches deliver acceptable performance in their respective applicable ranges, they retain the intrinsic limitations shared by conventional numerical discretization techniques, such as tedious mesh generation \cite{Hughes2005IGA}, the curse of dimensionality \cite{Yang2019WANs}, and rigid architectures hard to adapt to different dimensions and PDE types \cite{Xu2023M3AS}.

In recent years, neural network based approaches for solving partial differential equations (PDEs) — including physics-informed neural networks (PINNs) \cite{Raissi2019PINNs, Lu2021PINNHard} based on the general form of PDEs, the Deep Galerkin Method (DGM) \cite{Sirignano2018DGM}, Weak adversarial networks (WANs) \cite{Yang2019WANs}, Deep Ritz method (DRM) \cite{E2018DRM} based on the variational form of PDEs, as well as neural operators (FNO \cite{Li2020FNO} and DeepONet \cite{Lu2021DeepONet})—have emerged as promising mesh-free alternatives \cite{Brunton2024mlPDE}. These frameworks effectively circumvent the inherent mesh generation bottlenecks of traditional methods and hold significant potential for mitigating the curse of dimensionality. Nevertheless, they suffer from intrinsic issues such as gradient mismatch, gradient explosion, gradient vanishing \cite{Yu2022gPINN}, and poor fitting performance when solving PDEs with anisotropic or sharp solutions \cite{Calabro2021ELMSharp}. Several attempts have been made to address these challenges, such as Locally Adaptive Activation Functions (LAAF) \cite{Jagtap2020LAAF}, Physics-Informed Radial Basis Network (PIRBN)
\cite{Bai2023PIRBN}, adaptive symmetry-recomposition PINN (AsPINN) \cite{Liu2024ASPINN}, and Adaptive Basis-Inspired Deep Neural Network (ABI-DNN) \cite{Li2025ABIDNN}. Despite their remarkable improvements, these approaches still suffer from some limitations: model training requires solving highly nonlinear and nonconvex optimization problems \cite{Kiyani2025Optimizer}, which frequently introduce non-negligible optimization errors and degrade the accuracy of numerical solutions. Moreover, theoretical convergence guarantees for such training procedures are generally absent, hindering their stable application in high-precision scientific and engineering computations \cite{Arzani2023JCP}.

To mitigate the optimization errors of neural network-based PDE solvers, Randomized Neural Networks (RaNNs) \cite{Igelnik1995Stochastic} offer a promising strategy, reducing the originally nonlinear and nonconvex problem to a computationally efficient and accurate least-squares solution. Representative realizations of this family of methods include the Extreme Learning Machine (ELM) \cite{Huang2006ELM, Ren2025PIELM}, the local variant locELM \cite{Dong2021Local}, and the Random Feature Method (RFM) \cite{Chi2024RFMInterface}. Further developments hybridize RaNNs with subdomain-coupling strategies \cite{Li2023LRNN} and weak-form numerical frameworks, such as RaNN Petrov-Galerkin \cite{Shang2023Randomized} and Local RaNN-DG methods \cite{Sun2022LRNNDG}. Recently, Zhang et al. proposed a transferable neural network (TransNet) \cite{Zhang2024TransNet} based on basis function reparameterization, which increases the expressiveness of neural networks by uniformly sampling neurons and exhibits favorable approximation capacity and transferability for solving smooth PDEs. Collectively, these methodological advancements demonstrate that RaNN-based models can substantially alleviate the difficulty of optimization while preserving robust approximation performance for PDE solution tasks \cite{vanBeek2026LocalRFM}.

In spite of these merits, RaNNs struggle to construct informative, task-aligned fixed parameters for problems with anisotropic solutions or sharp solutions featuring steep gradients. Several adaptive network paradigms have been proposed to address this challenge. Dang et al. devised the Adaptive Growing Randomized Neural Network (AG-RaNN) \cite{Dang2024AGRaNN}, which constructively adds neurons based on residuals from preceding steps to dynamically adjust the network size during the iterative solution process, and achieves high‑precision solutions for PDEs with steep‑gradient solutions. Based on the similar residual guidance, Weng et al. developed the Deep Collocation Method (DCM) \cite{Weng2026DCM}, which optimizes sampling points and neural network basis function parameters via the iterative processes and achieves high‑precision solutions for boundary‑layer problems. Recently, Deng et al. further introduced the Adaptive Feature Capture Method (AFCM) \cite{Deng2025AFCM}, which adaptively redistributes feature functions and collocation points toward high-gradient regions to better resolve PDEs with near-singular solutions. Liu et al. put forward subspace neural networks (SNN) \cite{Xu2025SNN, Xu2025SNNweak}, a two-stage paradigm that first trains adaptive neural basis through lightweight optimization and then solves solution coefficients by least squares within the neural subspace, avoiding manually predefined basis functions and provides equation-adaptive feature spaces to deliver superior numerical accuracy for strongly anisotropic diffusion PDEs. Similar to the core idea of pre‑trained neural basis, DeePoly \cite{Liu2025DeePoly} proposed by Liu et al. first employs a DNN (Spotter) to capture complex global solution features and then conducts space optimization on the combined basis space spanned by piecewise polynomial functions (Sniper) and DNN latent features, circumventing exhaustive full‑network non‑convex training while inheriting rigorous high‑order convergence guarantees from polynomial approximation and achieving high‑precision solutions for high‑frequency PDEs. 

Despite the high accuracy demonstrated by the aforementioned methods in solving PDEs with sharp or discontinuous solutions, they still suffer from certain limitations—for instance: AG‑RaNN relies on the Fourier transform of the equation’s right‑hand‑side term for the initial selection of basis‑function parameters; AFCM hinges on domain decomposition and supports only subdomain-wise adaptation, maintaining a fixed number of neurons throughout iterations and thus leading to a limited accuracy upper bound and a lack of theoretical guarantees of convergence and error estimation; DCM, SNN and DeePoly employ optimizers to optimize basis‑function parameters, and this process incurs considerable computational overhead when solving partial differential equations with anisotropic or steep‑gradient solutions.

In this study, we propose Adaptive Transferable Neural Network (A-TransNet), a novel machine learning framework that adaptively expands neurons in regions with large solution gradients to effectively capture complex solution behaviors. Inspired by Adaptive Finite Element Methods (AFEMs) \cite{Kopteva2015AnisoAFEM, Cier2021AdaptiveFEM}, we perform gradient‑aware nonuniform construction on the reparameterized location parameters of neurons, and adaptively generate the corresponding shape parameters for each neuron according to the defined neuron density, thereby realizing neuron‑level adaptive generation of neurons. We further propose a neuron expansion iterative strategy to facilitate neuron‑level adaptivity with theoretical guarantees of convergence and error estimation. This framework achieves great performance for problems with anisotropic or sharp solutions.

\textbf{Main contributions.}
The main contributions of this paper are summarized as follows:
\begin{itemize}
	\item We propose A‑TransNet, which adopts an expansion-based adaptive neuron generation technique to realize neuron‑level adaptivity, delivering an effective and efficient approach for solving PDEs with anisotropic or sharp solutions.
	
	\item We present a gradient-aware construction strategy for nonuniform neuron location parameters, and further introduce a neuron density to achieve automatic nonuniform generation of shape parameters, enabling the network to effectively capture the anisotropic and sharp features of PDE solutions without extra computational costs.
	
	\item We rigorously establish convergence and a posteriori error estimates with a saturation plateau for the A-TransNet numerical solution, which lays a solid mathematical theoretical foundation for the proposed model.
	
	\item We demonstrate numerically that A-TransNet achieves excellent performance in solving representative PDEs with anisotropic or sharp solutions, outperforming state-of-the-art neural network methods in challenging scenarios, including the 3D anisotropic convection-diffusion equation and the 2D nonlinear Allen–Cahn equation with small parameters.
\end{itemize}

The remainder of the paper is organized as follows. Section~\ref{preliminaries} briefly reviews the TransNet method for solving PDEs, and two gradient-driven nonuniform neuron sampling strategies in the function approximation. Section~\ref{A-TransNet} presents the proposed adaptive transferable neural network. Section~\ref{sec:analysis} provides convergence analysis and error estimates. Section~\ref{Sec:NumerExper} presents numerical experiments to validate the effectiveness of the proposed method. Finally, Section~\ref{sec:sum} summarizes our research findings and elaborates on potential directions for future work.

\section{Preliminaries} \label{preliminaries}

In this section, the neural feature space, the transferable neural network (TransNet) \cite{Zhang2024TransNet}, and two gradient-driven nonuniform neuron sampling methods \cite{Pieper2024Nonuniform} are presented, which constitute the fundamental basis of the proposed A-TransNet.

\subsection{Neural Feature Space}

The neural feature space is an important concept in neural networks, referring to the functional space represented by the output of the last hidden layer, which is spanned by a set of neural basis functions (i.e., the neurons of the last hidden layer). For single‑hidden‑layer neural networks, the neural feature space can be expressed as
\begin{align} \label{feature space}
	\mathcal{P}_\text{NN} &= \operatorname{span}\left\{1,\ \sigma\left(\bw_1^\top \bx + b_1\right),\ \dots,\ \sigma\left(\bw_M^\top \bx + b_M\right)\right\},
\end{align}
where $\bx \in \mathbb{R}^d$ is the input to the neural network, $\sigma(\cdot)$ denotes the activation function, and $\bw_m \in \mathbb{R}^d,b_m \in \mathbb{R}$ ($m=1,2,\dots,M$) are the weight vector and bias term of the $m$-th neuron in the hidden layer, respectively.

Determining the neural feature space $\mathcal{P}_\text{NN}$ relies on proper sampling of $\{\boldsymbol{w}_m\}_{m=1}^M$ and $\{b_m\}_{m=1}^M$, which significantly affects the solution accuracy. Several popular existing paradigms are presented below. The Gaussian random feature model \cite{Li2019Towards, Nelsen2021Random} samples parameters based on Gaussian distribution and achieves satisfactory performance for smooth problems. The uniform-distributed feature model \cite{Dang2024AGRaNN} determines the sampling range by performing Fourier transforms on the right-hand side of the governing equation, thereby further improving the accuracy. The optimizable-distribution feature model \cite{Weng2026DCM} dynamically tunes random parameters via the Adam optimizer guided by residual information, yielding acceptable accuracy for complex PDE problems. Furthermore, He et al. investigate the influence of diverse activation functions on neural feature spaces \cite{Yang2026FEN} and propose trigonometric activation functions to enrich neural feature spaces, enabling better approximation of high-frequency solution characteristics in PDEs. Therefore, selecting appropriate parameters for the neural feature space is particularly critical when solving PDEs with anisotropic or sharp solutions.

The output function of the neural network $u_\text{NN}(\bx) \in \mathcal{P}_\text{NN}$ is typically takes the form of a linear combination of neural basis functions
\begin{align} \label{Transnet Solution}
	u_\text{NN}(\bx) &= \sum_{m=1}^M \alpha_m \sigma\left(\bw_m^\top \bx + b_m\right) + \alpha_0,
\end{align}
where $\alpha_0, \alpha_1, \dots, \alpha_M$ denote the parameters of the output layer. 

\subsection{Transferable Neural Network}\label{TransNet}

TransNet, proposed in \cite{Zhang2024TransNet}, is a single-hidden-layer fully connected neural network for solving PDEs of the following form:
\begin{align} \label{PDE}
	\begin{aligned}
		\mathcal{L}(u(\bx)) &= f(\bx), \quad \bx \in \Omega, \\
		\mathcal{B}(u(\bx)) &= g(\bx), \quad \bx \in \partial\Omega,
	\end{aligned}
\end{align}
where $u$ denotes the unknown solution of the PDE, $\Omega \subset \mathbb{R}^d$ is a bounded domain with boundary $\partial\Omega$, $\mathcal{L}(\cdot)$ is a differential operator, $\mathcal{B}(\cdot)$ is the operator defining the boundary conditions, and $f$ and $g$ are the right-hand side terms associated with $\mathcal{L}(\cdot)$ and $\mathcal{B}(\cdot)$, respectively.

In TransNet, the neural basis functions are reparameterized within the neural feature space. In particular, each neural basis function in the neural feature space can be regarded as a partition hyperplane:
\begin{align*}
	\bw_m^\top \bx + b_m = 0.
\end{align*}
In the geometric space, this expression can be rewritten as
\begin{align} \label{rewritten hyperplane}
	\gamma_m \bigl(\ba_m^\top \bx + r_m\bigr) = 0,
\end{align}
where $\gamma_m \in \mathbb{R}^+$, $\ba_m \in \mathbb{R}^d$ with $\|\ba_m\|_2 = 1$, and $r_m \in \mathbb{R}$. Analogous to the point-normal form of hyperplane equations in $\mathbb{R}^d$, the vector $\ba_m$ and the scalar $r_m$ of the partition hyperplane \eqref{rewritten hyperplane}, which respectively dictate the hyperplane orientation and represent the distance from the origin to this hyperplane, are collectively referred to as location parameters. In contrast, the scalar $\gamma_m$ is termed the shape parameter due to its association with the geometric shape characteristics of the neural basis function, which quantifies its transition rate. The visualization of shape parameters with the $\tanh(x)$ activation function is presented in Figure \ref{fig:shape parameters}. Clearly, the relationship between the original parameters $\bw_m, b_m$ and the reparameterized parameters $\gamma_m, \ba_m, r_m$ is given as follows:
\begin{align*}
	\bw_m = \gamma_m \ba_m, \quad b_m = \gamma_m r_m.
\end{align*}
In this way, the solution \eqref{Transnet Solution} can be rewritten as
\begin{align} \label{rewritten solution}
	u_{\rm NN}(\bx) = \sum_{m=1}^M \alpha_m \sigma\bigl(\gamma_m \bigl(\ba_m^\top \bx + r_m\bigr)\bigr) + \alpha_0.
\end{align}

\begin{figure}[!htb]
	\centering
	\vspace{-0.2cm}
	\includegraphics[width=0.98\textwidth]{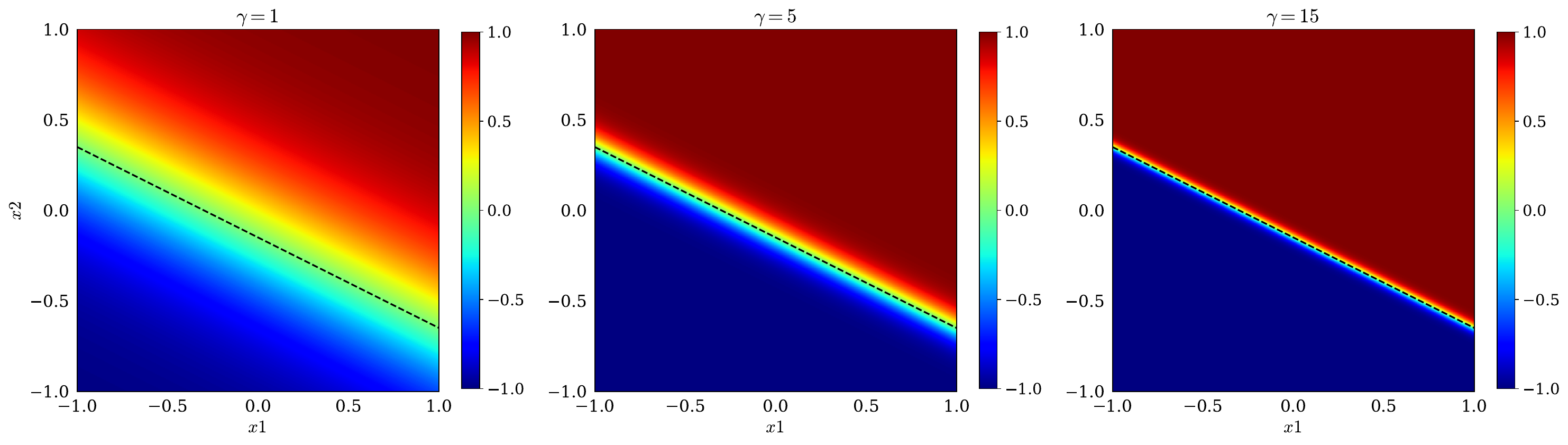}
	\vspace{-0.3cm}
	\caption{Visualization of hyperplanes corresponding to different shape parameter $\gamma$ in two-dimensional space.}
	\label{fig:shape parameters}
	\vspace{-0.25cm}
\end{figure}

The location parameters $\{\ba_m\}_{m=1}^M$ can be sampled from a $d$-dimensional standard Gaussian distribution in the Cartesian coordinate system, and the sampled results are then normalized to unit vectors, while $\{r_m\}_{m=1}^M$ can be generated from a uniform distribution over the interval $[0,1]$. As established in \cite{Zhang2024TransNet}, the neurons generated via this scheme are uniformly distributed in the unit ball $B_1(\mathbf{0})$. For the sake of model simplicity and generalization performance, all neurons in the hidden layers of TransNet share a single shape parameter in \cite{Zhang2024TransNet}, i.e., $\gamma_m = \gamma$ for $m = 1, 2, \dots, M$. Gaussian Random Fields (GRFs) are introduced to generate auxiliary functions for the offline pre-tuning of $\gamma$ based on function approximation capability, and grid search is employed to find the optimal value of $\gamma$ that minimizes the average approximation error of all auxiliary functions. This method achieves excellent accuracy and generalization ability for solving smooth PDE problems, yet it exhibits poor performance for challenging problems with anisotropic or sharp solutions. This limitation essentially stems from the inherent uniform sampling strategy, highlighting the necessity of developing an adaptive parameter selection scheme.

With the location and shape parameters determined, the neural feature space \eqref{feature space} is fully constructed offline, at which point the TransNet solution can be expressed as \eqref{rewritten solution}. The physics-informed loss function is constructed as follows:
\begin{align} \label{loss function}
	\begin{split}
	\hat{Loss}(u) = \hat{Loss}(\boldsymbol{\alpha})
	= \frac{\lambda_I}{J_I} \sum_{i=1}^{J_I} | \mathcal{L}\bigl(u_{\text{NN}}(\bx_i^I)\bigr) - f(\bx_i^I) |^2 + \frac{\lambda_B}{J_B} \sum_{i=1}^{J_B} | \mathcal{B}\bigl(u_{\text{NN}}(\bx_i^B)\bigr) - g(\bx_i^B) |^2,
	\end{split}
\end{align}
where $\lambda_I$ and $\lambda_B$ are predefined penalty parameters, which are set to 1 throughout this paper unless otherwise stated, and $J_I$ and $J_B$ denote the numbers of sampling points in the interior and on the boundary of the computational domain, respectively. Finally, minimizing the loss function \eqref{loss function} yields the output layer parameters $\{\alpha_m\}_{m=0}^M$, which is essentially a linear or nonlinear least squares problem. With these parameters solved, the TransNet solution to the PDE is thereby obtained. The procedure of TransNet is presented in Figure \ref{fig:Transnet flowchart}.

\begin{figure}[!htb]
	\centering
	\vspace{-0.25cm}
	\includegraphics[width=0.94\textwidth]{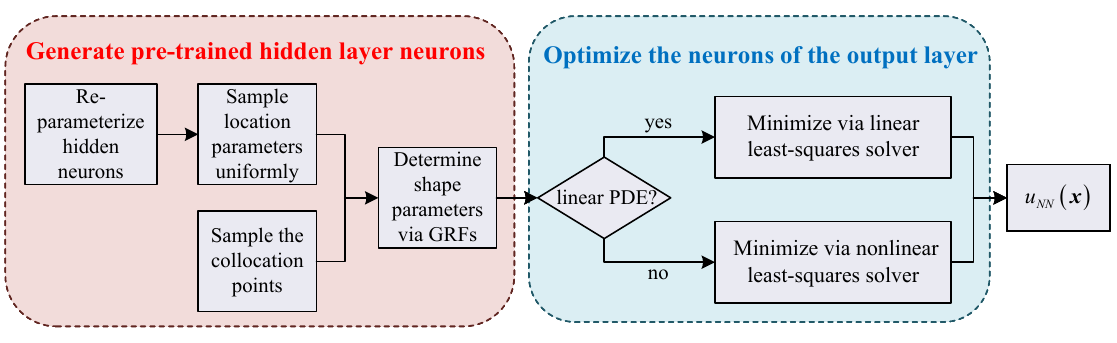}
	\vspace{-0.3cm}
	\caption{Flowchart of the TransNet algorithm.}
	\label{fig:Transnet flowchart}
	\vspace{-0.25cm}
\end{figure}

\subsection{Two Gradient-Driven Nonuniform Neuron Sampling Methods}

RaNNs are also applicable to function approximation, where the approximate function takes the same form as \eqref{Transnet Solution}, with only the loss function adjusted to the following form:
\begin{align*}
	\begin{split}
	 f(\boldsymbol{x}) = \sum_{m=1}^{M} \alpha_m \sigma\bigl(\boldsymbol{\omega}_m^\top \boldsymbol{x} + b_m\bigr) + \alpha_0, \quad L(f) = \frac{1}{K} \sum_{k=1}^{K} \bigl(f(\bx_k) - y_k\bigr)^2
	\end{split}
\end{align*}
where $K$ is the number of sampling points, with $y_k$ being the observed value at $\bx_k$.

Active subspaces \cite{Constantine2014Active, Constantine2015CAS} have been well‑established and widely applied in computational mathematics. In \cite{Pieper2024Nonuniform}, the active subspaces method is adopted for neuron sampling in function approximation, which is capable of capturing the anisotropic characteristics of the target function, as summarized in Algorithm~\ref{active subspace method}.
\vspace{-0.2cm}
\begin{algorithm}[!htb]
	\caption{Active Subspaces Method}
	\label{active subspace method}
	\KwIn{Gradient vectors $\{\bm{g}_j\}_{j=1}^{J_I}$.}
	\KwOut{Neural feature space parameters $\{\boldsymbol{w}_m\}_{m=1}^M$ and $\{b_m\}_{m=1}^M$.}
	
	Perform singular value decomposition on the gradient matrix, yielding $G = [\bm{g}_j] = V \Sigma W$.
	
	Sample new gradients from the active subspaces via $\widetilde{\bm{g}}_m = V \xi_m$, where $\xi_m \sim \mathcal{N}(0, \Sigma)$.
	
	Normalize $\widetilde{\bm{g}}_m$ to obtain the weight vector: $ \boldsymbol{\omega_m} = \widetilde{\bm{g}}_m / \|\widetilde{\bm{g}}_m\|_2$.
	
	Sample the bias $b_m$ uniformly over a bounded interval: $b_m = - \boldsymbol{\omega_m} \cdot \bx_m$, where $\bx_m$ is uniformly sampled in the computational domain $\Omega$.
	
	Return neural feature space parameters $\{\boldsymbol{w}_m\}_{m=1}^M$ and $\{b_m\}_{m=1}^M$.
\end{algorithm}
\vspace{-0.2cm}

To capture the localized sharp features of the function, the local gradient method was also proposed in \cite{Pieper2024Nonuniform}, as illustrated in Algorithm~\ref{local gradient method}.

\vspace{-0.2cm}
\begin{algorithm}[!htb]
	\caption{Local Gradient Method}
	\label{local gradient method}
	\KwIn{Sampling points $\{\bx_j\}_{j=1}^{J_I}$ and corresponding gradient vectors $\{\bm{g}_j\}_{j=1}^{J_I}$.}
	\KwOut{Neural feature space parameters $\{\boldsymbol{w}_m\}_{m=1}^M$ and $\{b_m\}_{m=1}^M$.}
	
	Select collocation points $\{\bx_m\}_{m=1}^{M}$ from sampling points with probabilities proportional to the gradient magnitude $\left\| \bm{g}_j \right\|$, whose associated gradient is $\{\bm{g}_m\}_{m=1}^{M}$.
	
	Set its normalized gradient as the weight vector: $\boldsymbol{\omega}_m = \bm{g}_m / \|\bm{g}_m\|_2$.
	
	Compute the bias such that the hyperplane passes near the selected point: $b_m = -\boldsymbol{\omega}_m \cdot \bm{x}_m$.
	
	Return neural feature space parameters $\{\boldsymbol{w}_m\}_{m=1}^M$ and $\{b_m\}_{m=1}^M$.
\end{algorithm}
\vspace{-0.2cm}

The performance of the active subspaces method is demonstrated in Figures~\ref{fig:active}, consistent with \cite{Pieper2024Nonuniform}. For the target function $f(x_1,x_2)=\sin(5x_1-5\sqrt{2}x_2)$, it exhibits anisotropy along the direction $(1, -\sqrt{2})$, as shown in the left subfigure. The Gaussian random feature method yields relatively uniform sampling directions and only achieves an accuracy of 8.65e-02 with 300 neurons. By contrast, the active subspaces method well captures the anisotropy of the target function and attains an accuracy of 1.47e-03. As for the local gradient method, its results are presented in Figure~\ref{fig:local} as in \cite{Pieper2024Nonuniform}. For the target function $f(x)=e^{-30x^2}$, the local gradient method distributes more neurons over regions with large gradients, reaching an accuracy of 1.68e-03 using merely 100 neurons, which represents an improvement over the accuracy of 6.23e-03 obtained by the uniform sampling method. For detailed theoretical analysis and numerical results, see \cite{Pieper2024Nonuniform}.

\begin{figure}[!htb]
	\centering
	\vspace{-0.2cm}
	\includegraphics[width=0.98\textwidth]{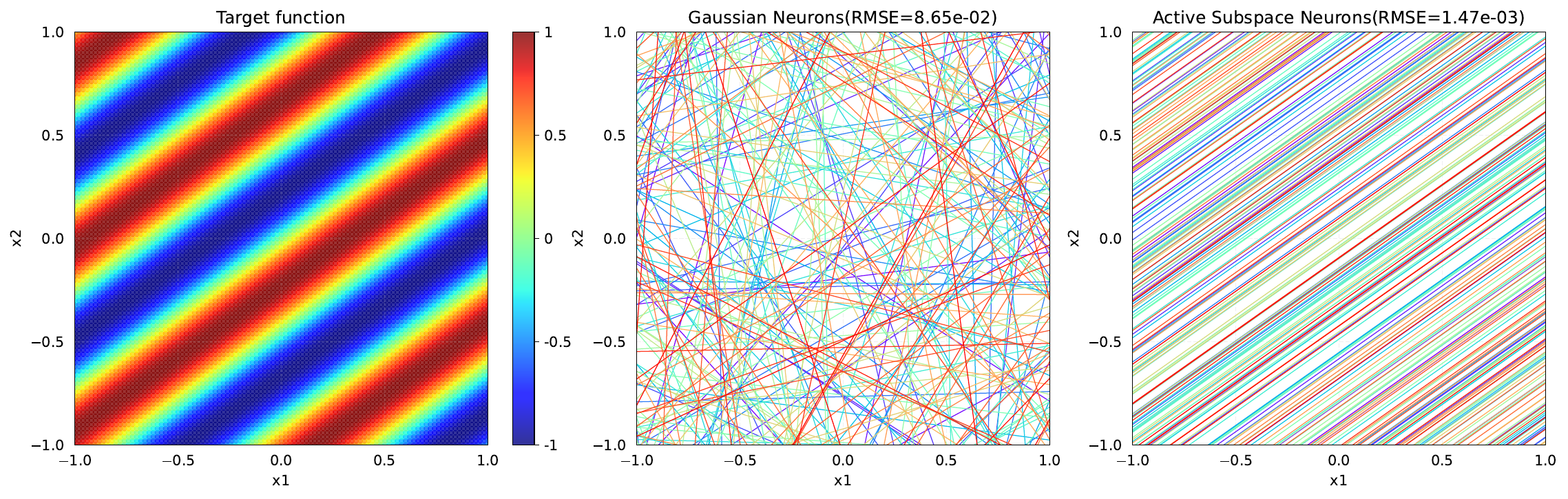}
	\vspace{-0.4cm}
	\caption{Neuron visualization of the active subspaces method. $f(x_1,x_2)=sin(5x_1-5\sqrt{2}x_2)$.}
	\label{fig:active}
\end{figure}

\begin{figure}[!htb]
	\centering
	\vspace{-0.1cm}
	\includegraphics[width=0.68\textwidth]{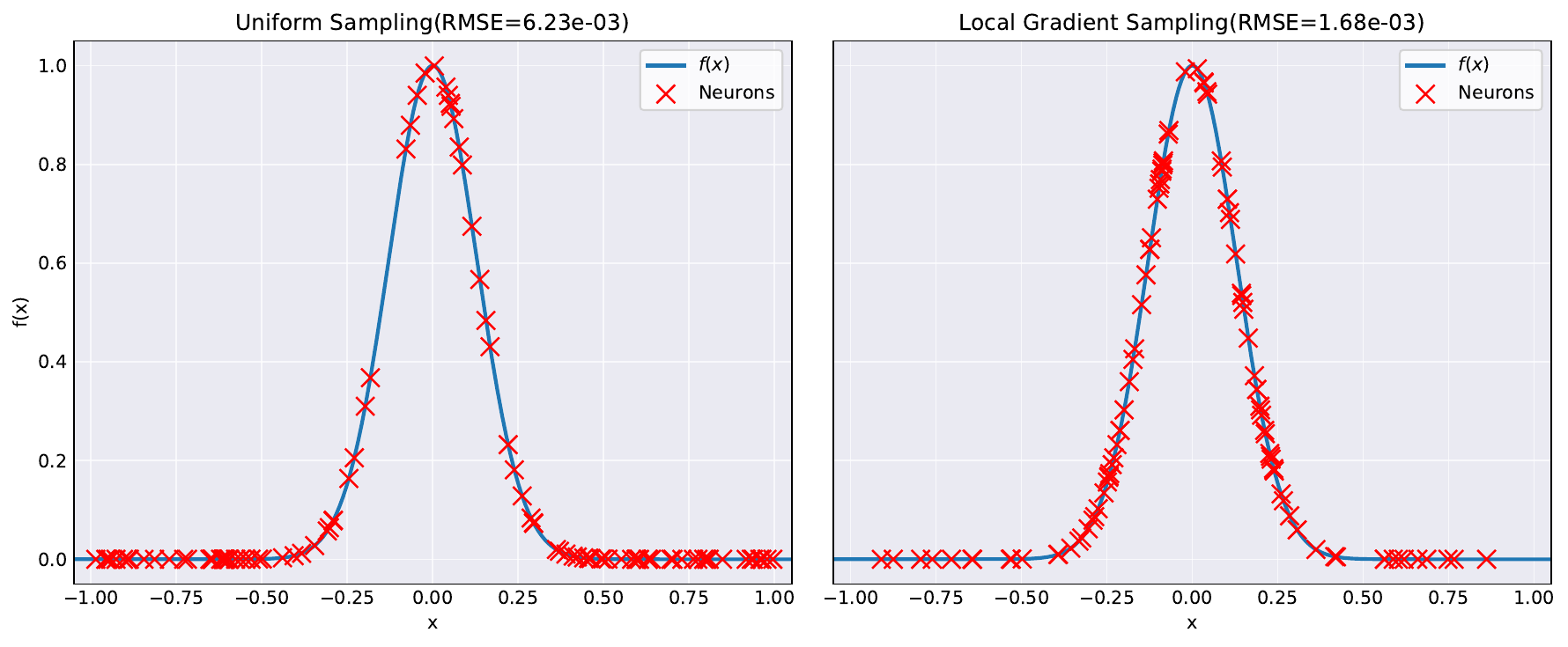}
	\vspace{-0.45cm}
	\caption{Neuron visualization of the local gradient method. $f(x)=e^{-30x^2}$.}
	\label{fig:local}
	\vspace{-0.25cm}
\end{figure}

In brief, the active subspaces method captures anisotropic structures, while the local gradient method focuses on local function behaviors, enabling more targeted function approximation in different scenarios. Their application to numerical PDE solving and combination for problems featuring both anisotropy and sharp solutions will be detailed in Section~\ref{A-TransNet}.

\section{Adaptive Transferable Neural Network} \label{A-TransNet}

In this section, we present the Adaptive Transferable Neural Network (A-TransNet) method for solving PDEs with anisotropic or sharp solutions. The core idea of A-TransNet is to adaptively select network parameters based on the gradient information of approximate solutions, enabling the network to capture the anisotropy or sharp features of the solutions. Let us consider the general partial differential equation in the form of \eqref{PDE}, where the form of solution is given by \eqref{rewritten solution}. The flowchart illustrating the A-TransNet method is presented in Figure \ref{fig:A-TransNet flowchart}.

\begin{figure}[!htb]
	\centering
	\vspace{-0.2cm}
	\includegraphics[width=0.99\textwidth]{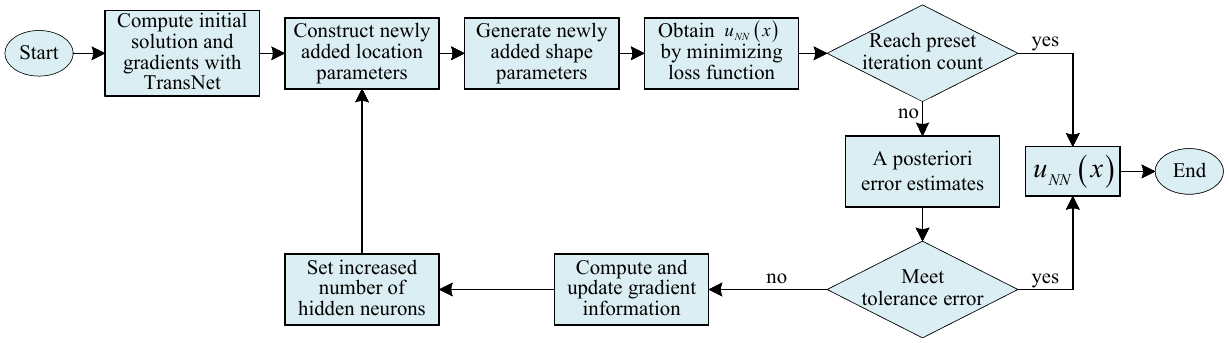}
	\vspace{-0.3cm}
	\caption{Flowchart of the A-TransNet algorithm.}
	\label{fig:A-TransNet flowchart}
	\vspace{-0.2cm}
\end{figure}

First, the initial solution is computed via a small‑scale TransNet to obtain approximate gradient information. Next, the hidden‑layer neurons are determined using such gradient information, including the construction of position parameters and the generation of shape parameters. Then, the outer‑layer coefficients are solved via the least‑squares method. Herein, the a posteriori error is estimated as the difference between the neural network solutions of two successive steps, inspired by Corollary~\ref{a posteriori error estimates of A-TransNet} in subsection~\ref{4.2}. Once the preset iteration count is reached or this difference falls below the preset tolerance, the current iterative numerical solution $u_{\rm NN}(\bx)$ is output and the procedure terminates. Otherwise, compute and update the gradient information based on the current numerical solution, set the increased number of hidden neurons, and proceed to the next iteration.

Next, we elaborate on the gradient-driven location parameters construction strategy, the shape parameters generation strategy, and the iterative strategy for obtaining approximate gradient information in A-TransNet in detail.

\subsection{Nonuniform Construction Strategy for Location Parameters} \label{location parameters}

A-TransNet adopts a hybrid construction strategy for location parameters to better capture the anisotropy and sharp features of the solution. $M_{ani}$ neurons have their location parameters sampled via the active subspaces method to capture solution anisotropy, while $M_{sharp}$ neurons are sampled via the local gradient method to capture solution sharpness. Let $M$ denote the total number of neurons to be constructed, so that $M=M_{ani}+M_{sharp}$. Meanwhile, we define $\mu=M_{ani}/M$ as the allocation parameter.

Consider the unit ball computational domain $B_1(\mathbf{0})$, let $\bm{g}_j$ be an approximation to the gradient of the solution $\nabla u(\bm{x}_j)$, where $\bm{x}_j$ are training points uniformly sampled in the computational domain $\Omega$. Analogous to the active subspaces method, we leverage the properties of Gaussian random variables to conduct singular value decomposition on the gradient matrix:
\begin{align} \label{SVD}
	G = [\bm{g}_j]_{j=1,2,\dots,J_I} = V \Sigma W, \quad \text{with } V \in \mathbb{R}^{d \times r},\ \Sigma \in \mathbb{R}^{r \times r},\ \text{and } W \in \mathbb{R}^{r \times J_I}.
\end{align}
Subsequently, we generate new gradient samples from the active subspaces:
\begin{align} \label{new sample}
	\widetilde{\bm{g}}_m = V \xi_m, \quad \text{where } \xi_m \sim \mathcal{N}(0, \Sigma) \text{ and } m=1,2,\dots,M_{ani}.
\end{align}
Finally, we normalize these samples to obtain $\ba_m$, followed by uniform sampling to generate $r_m$:
\begin{align} \label{get}
	\ba_m = \widetilde{\bm{g}}_m / \|\widetilde{\bm{g}}_m\|_2 \in \mathbb{S}^{d-1}, \quad r_m \sim U[-1, 1].
\end{align}
It should be noted that $r_m$ is selected by sampling from a uniform distribution, which yields better numerical performance compared with the sampling scheme in the original active subspaces method. This is probably because sampling $r_m$ from the uniform distribution after determining the direction $\ba_m$ makes the neurons more uniformly distributed within the computational domain \cite{Zhang2024TransNet}. The above procedure constructs the location parameters for the $M_1$ neurons, which are thereby equipped to effectively capture the anisotropic features inherent in the solution.

\begin{rmk}    
	The above describes the construction method for the location parameters of $M_{ani}$ neurons within the unit ball $B_1(\mathbf{0})$ as the computational domain. When the computational domain $\Omega$ takes other shapes, simple translation and dilation operations on neurons can be implemented to cover the entire computational domain, i.e., $\boldsymbol{\ba}'_m = \frac{\boldsymbol{\ba}_m}{\boldsymbol{s}}$, $r'_m = r_m - \boldsymbol{\ba}_m^T \frac{\boldsymbol{c}}{\boldsymbol{s}}$, where $\boldsymbol{s}$ and $\boldsymbol{c}$ denote the radius and center of $\Omega$, respectively. In particular, if $\Omega$ is centered at the origin, we can always find a ball of radius $R$ enclosing the computational domain such that $\Omega \subset B_R(\mathbf{0})$. In this case, the parameter $r$ follows a uniform distribution $r \sim U(-R, R)$. It is worth noting that $R$ should be chosen slightly larger than the minimal value required to exactly cover $\Omega$. For instance, for a square domain with side length $2$, one can set $R = 1.5$.
\end{rmk}

Following the local gradient method in function approximation, $M_{sharp}$ neurons are sampled as follows to capture the sharp features of the function effectively. We first select the top $M_{sharp}$ training points $\bx_m$ with the largest gradient magnitudes $\|\bm{g}_j\|$ ($\bm{g}_j \approx \nabla u(\bm{x}_j)$), and then the gradients of these candidates are normalized to obtain $\ba_m$:
\begin{align} \label{normalize}
	\ba_m = \bm{g}_m / \|\bm{g}_m\|_2, \quad m = M_{ani} + 1,M_{ani} + 2, \dots, M.
\end{align}
Lastly, $r_m$ is sampled to ensure that the hyperplane \eqref{rewritten hyperplane} determined by $\ba_m$ and $r_m$ passes through the chosen sampling points:
\begin{align} \label{r}
	r_m = -\ba_m \cdot \bm{x}_m, \quad m = M_{ani} + 1, M_{ani} + 2, \dots, M.
\end{align}
In this manner, the hyperplanes defined by the location parameters of the remaining $M_{sharp}$ neurons pass through the  training points with the largest gradient magnitudes, rendering these neurons more sensitive to regions with large gradients and thus enabling them to capture the sharp features of the solution.

To summarize the above procedure for constructing the location parameters, we present it concisely as Algorithm~\ref{location parameters algorithm} below.

\begin{algorithm}[!htb]
	\caption{Location Parameters Construction Algorithm}
	\label{location parameters algorithm}
	\KwIn{Radius of computational domain $R$, total number of neurons $M$, allocation parameter $\mu$, approximate gradient of solution $\{\boldsymbol{g}_i\}_{i=1}^{J_I}$, training points $\{\bx_j\}_{j=1}^{J_I}$.}
	\KwOut{Location parameters $\{\ba_m\}_{m=1}^M$ and $\{r_m\}_{m=1}^M$ of $M$ neurons.}
	
	Divide the total $M$ neurons into two disjoint subsets:  $M_{ani} = \left\lceil \mu M\right\rceil$ neurons to capture anisotropic features and $M_{sharp} = M - M_{ani}$ neurons to capture sharp features.
	
	Perform singular value decomposition (SVD) on the gradient matrix as specified in \eqref{SVD}.
	
	Generate new gradient samples $\widetilde{\bm{g}}_m$ as in \eqref{new sample} from the active subspaces.
	
	Normalize the samples and obtain the location parameters in accordance with \eqref{get}.
	
	Select the top $M_{sharp}$ training points (not previously selected) with the largest gradient magnitudes $\|\bm{g}_m\|$.
	
	Normalize the gradients of selected points as shown in \eqref{normalize} to obtain $\ba_m$.
	
	Sample $r_m$ as \eqref{r} such that the hyperplane defined by $\bm{a}_m$ and $r_m$ passes through the selected sampling points.
	
	Return location parameters $\{\ba_m\}_{m=1}^M$ and $\{r_m\}_{m=1}^M$ of all $M$ neurons.
\end{algorithm}

\begin{rmk}
	\vspace{-0.3cm}    
	The allocation parameter $\mu$ determines the sensitivity of all neurons to the anisotropy and sharp features of the solution. For PDEs with anisotropic solutions, $\mu$ should be moderately increased, whereas for equations featuring sharp solutions, $\mu$ ought to be appropriately reduced. Without prior information on solution characteristics, a medium value of $\mu$ such as $2/3$ is recommended, since the computational accuracy is insensitive to variations of $\mu$ within a reasonable range, as demonstrated in Subsection~\ref{Subsec:5.8}.
\end{rmk}

\subsection{Generation Strategy for Nonuniform Shape Parameters} \label{shape parameters}

\subsubsection{neuron density}

The adaptive characteristics of location parameters imply that adopting adaptive shape parameters for all neurons may yield improved performance of the neural network. Inspired by the specific geometric interpretation of shape parameters and following the principle that larger gradients at neuron positions correspond to larger shape parameters, we design a shape parameter generation method for adaptive location parameters. Accordingly, we introduce the concept of neuron density.

Consider the neuron partition hyperplane $\ba_m^\top \bm{x} + r_m = 0$, where $\boldsymbol{a}_m = (a_{m1}, a_{m2}, \dots, a_{md})$. Perform the following coordinate transformation:
\begin{align} \label{coordinate transformation}
	\begin{cases}
		\theta_{mi} = \arctan\left(\frac{a_{m,i+1}}{a_{m1}}\right) \in (-\frac{\pi}{2}, \frac{\pi}{2}), \\
		l_m = r_m \in [-R, R],
	\end{cases} \quad \text{where } i = 1,2,\dots,d-1.
\end{align}
Then the concept of neuron density can be difined in Definition~\ref{neuron density}.

\begin{df}[Neuron Density] \label{neuron density}
	The neuron density is defined by
	\begin{align*}
		\rho_m = \frac{N_m}{V_m}, \quad m = 1, 2, \dots, M,
	\end{align*}
	where $N_m$ denotes the number of points counted within its neighborhood $\prod_{i=1}^{d-1} \left[\theta_{mi}-\frac{\pi \delta}{2R},\, \theta_{mi} + \frac{\pi \delta}{2R}\right] \times [l_m-\delta,\, l_m +\delta]$, and $V_m$ is the volume of this neighborhood (using such a rectangular neighborhood alleviates scale‑related issues and leads to more accurate computations). 
\end{df}

\begin{rmk}    
	To guarantee the accuracy of neuron density computation, when the neighborhood of the neuron
	$\prod_{i=1}^{d-1} \left[\theta_{mi}-\frac{\pi \delta}{2R},\, \theta_{mi} + \frac{\pi \delta}{2R}\right] \times [l_m-\delta,\, l_m +\delta]$
	extends beyond the transformed computational domain region $\prod_{i=1}^{d-1} \left(-\frac{\pi}{2}, \frac{\pi}{2}\right) \times [-R, R]$,
	\(V_m\) is defined as the volume of the intersection of these two regions. Correspondingly, \(N_m\) denotes the number of points counted within this intersection.
\end{rmk}

In the one-dimensional case, the hyperplane \eqref{rewritten hyperplane} defined by a neuron reduces to a single point, and the density can be obtained approximately by counting the number $N_m$ of points within the neighborhood $[-\frac{r_m}{a_{m1}}-\delta, -\frac{r_m}{a_{m1}}+\delta]$, where $\ba_m = (a_{m1})$:
\begin{align*}
	\rho_m = \frac{N_m}{2\delta}, \quad m = 1, 2, \dots, M.
\end{align*}
Considering the two-dimensional case, the hyperplane determined by each neuron becomes a straight line, and we first perform the following coordinate transformation on it:
\begin{align*}
	\begin{cases}
		\theta_{m1} = \arctan\left(\frac{a_{m2}}{a_{m1}}\right) \in (-\frac{\pi}{2}, \frac{\pi}{2}), \\
		l_m = r_m \in [-R, R],
	\end{cases} \quad \text{where } \boldsymbol{a}_m = (a_{m1}, a_{m2}).
\end{align*}
Through the above coordinate transformation, we represent the hyperplane as a point in a pseudo-polar coordinate system, and the neuron density can be approximately calculated by computing the number of points within the rectangular neighborhood $\left[\theta_{m1} -\frac{\pi\delta}{2R}, \theta_{m1} + \frac{\pi\delta}{2R}\right] \times [l_m-\delta, l_m+\delta]$:
\begin{align*}
	\rho_m = \frac{N_m R}{2 \pi \delta^2}, \quad m = 1, 2, \dots, M.
\end{align*}
By analogy, extending to the $d$-dimensional case, after the coordinate transformation \eqref{coordinate transformation} and counting the number $N_m$ within its neighborhood $\prod_{i=1}^{d-1} \left[\theta_{mi}-\frac{\pi \delta}{2R},\, \theta_{mi} + \frac{\pi \delta}{2R}\right] \times [l_m-\delta,\, l_m +\delta]$, the neuron density can be defined by
\begin{align} \label{density formula}
	\rho_m = \frac{N_m R^{d-1}}{2 \pi^{d-1} \delta^d}, \quad m = 1, 2, \dots, M.
\end{align}

\subsubsection{Generation Strategy}
Through extensive observations and experimental validation, an empirical formula for the shape parameter $\gamma$ is proposed in \cite{Lu2025MTransNet}:
\begin{align} \label{empirical formula}
	\gamma \approx C \frac{M^{1/d}}{R},
\end{align}
where $C$ is an empirical constant, which typically depends on the specific form and parameters of the target partial differential equation (PDE), and $d$ denotes the dimension. Using the above empirical formula, the empirical constant $C$ can be determined via the golden-section search method with a small value of $M$ (on a small scale), thereby predicting an appropriate value of the shape parameter $\gamma$ for a new (typically larger) $M$ with low computational cost.

Assume that the number of neurons required to achieve the neuron density \eqref{density formula} over the entire domain transformed via \eqref{coordinate transformation} is given by
\begin{align*}
	M_m = \rho_m \cdot \text{Vol}(D) = \frac{N_m R^{d-1}}{2 \pi^{d-1} \delta^d} \cdot 2 \pi^{d-1} R = \frac{N_m R^d}{\delta^d},
\end{align*}
where $D = \prod_{i=1}^{d-1} (-\frac{\pi}{2}, \frac{\pi}{2}) \times [-R,R]$ and $\text{Vol}(D)$ denotes the volume of the entire transformed domain $D$. Substituting $M_m$ into \eqref{empirical formula} yields the generation formula:
\begin{align} \label{shape parameters formula}
	\gamma_m \approx C \frac{N_m^{1/d}}{\delta}, \quad i = 1, 2, \dots, M,
\end{align}
where $\delta$ is the preset neighborhood radius, and $N_m$ denotes the number of neurons within the neighborhood. With the above theoretical analysis and discussion, we summarize the corresponding algorithm as Algorithm~\ref{shape parameters algorithm}.
\vspace{-0.2cm}
\begin{algorithm}[!htbp]
	\caption{Shape Parameters Generation Algorithm}
	\label{shape parameters algorithm}
	\KwIn{Location parameters $\{\ba_m\}_{m=1}^M$ and $\{r_m\}_{m=1}^M$, radius of computational domain $R$, neighborhood radius $\delta$, empirical constant $C$.}
	\KwOut{Shape parameters $\{\gamma_m\}_{m=1}^M$ of $M$ neurons.}
	
	Apply the coordinate transformation given in \eqref{coordinate transformation} to each position parameter.
	
	Calculate the number of points $N_m$ in the neighborhood of each neuron under the new coordinate system using methods such as KD-Tree.
	
	Compute the shape parameter $\gamma_m$ for each neuron using the generation formula \eqref{shape parameters formula}.
	
	Return shape parameters $\{\gamma_m\}_{m=1}^M$ of all $M$ neurons.
\end{algorithm}

\begin{rmk}    
	In the generation formula \eqref{shape parameters formula}, the magnitude of the parameter $\delta$ determines the sensitivity of the shape parameter to the position parameters. According to experimental observations, its selection may be jointly determined by the computational domain radius $R$, the total number of hidden-layer neurons $M$, and the characteristics of the partial differential equation. We do not discuss this here, and in the numerical experiments of Section~\ref{Sec:NumerExper}, we directly set it as $\delta = \frac{3R}{M^{1/d}}$ to achieve higher solution accuracy. 
\end{rmk}

Theorem~\ref{neuron density consistency} establishes the consistency between the proposed generation formula and the empirical formula in \cite{Lu2025MTransNet} under uniform neuron distribution, which demonstrates the rationality of the proposed generation formula.

\begin{thm}[Consistency of Neuron Density and Shape Parameters]\label{neuron density consistency}
	Assume that the neuron location parameters $\{\ba_m\}_{m=1}^M$ and $\{r_m\}_{m=1}^M$ are uniformly sampled as in conventional TransNet. For sufficiently large neuron number $M$ and sufficiently small neighborhood scale $\delta$, the expectation of the neuron density satisfies
	\begin{align*}
		\mathbb{E}[\rho_m] = \frac{M}{2\pi^{d-1} R}.
	\end{align*}
	Furthermore, the expectation of the generated shape parameters satisfies
	\begin{align*}
	\mathbb {E}[\gamma_m] \approx  C \frac {M^{1/d}}{R}.
	\end{align*}
\end{thm}
\begin{proof}
	The proof is provided in \ref{app:proofthm2}.
\end{proof}

As a numerical illustration, we refer to the Poisson equation in Subsection~\ref{Subsec:5.3}, whose true solution is showed in Figure~\ref{fig:5.3showneurons} (left). We set $\delta=0.2$ and calculate the neuron density according to \eqref{density formula}, yielding the density expectation $\mathbb{E}[\rho_i]=430.46$ and the coefficient of variation $\text{CV}=0.11<0.15$, which validates the rationality of the proposed neuron density definition. The shape parameter visualization of the additional 4000 neurons generated during the A-TransNet iteration is presented in Figure~\ref{fig:5.3showneurons} (right). More neurons are concentrated in the regions with rapid solution variations, accompanied by a corresponding increase in shape parameters, which greatly enhances the ability of the neural network to capture sharp changes in the solution.

\begin{figure}[!htb]
	\centering
	\vspace{-0.25cm}
	\includegraphics[width=0.75\textwidth]{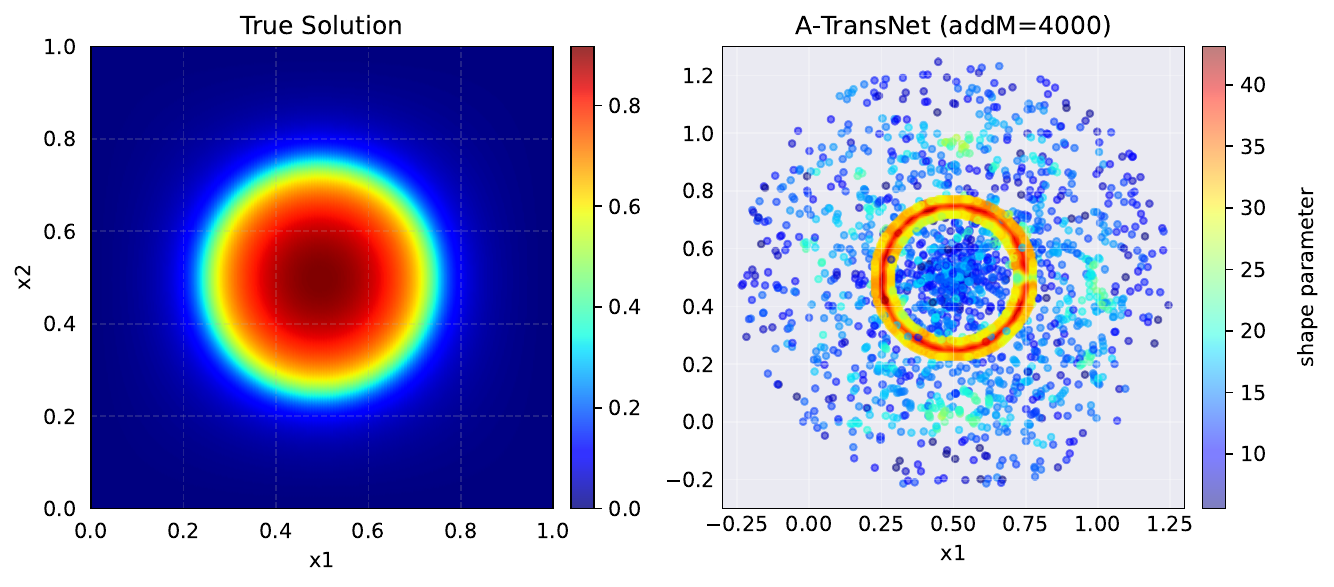}
	\vspace{-0.35cm}
	\caption{Shape parameter visualization for example in Subsection~\ref{Subsec:5.3} (neurons represented by their foot points)}
	\label{fig:5.3showneurons}
	\vspace{-0.2cm}
\end{figure}

\subsection{A-TransNet Algorithm} \label{Sec:layer-grow}

The methods described in Sections \ref{location parameters} and \ref{shape parameters} for determining hidden-layer neurons require approximate derivative information of the solution, which is generally not available in problem \eqref{PDE}. At the $s$-th iteration, approximate derivatives of the exact solution can be obtained from the approximate solution of the previous step, i.e.,
\begin{align} \label{approximate derivatives}
	\nabla_{\boldsymbol{x}} u^{s-1}_{\text{NN}}(\boldsymbol{x}) = \sum_{m=1}^{M_{s-1}} \alpha_m \gamma_m \sigma' \bigl( \gamma_m \bigl( \boldsymbol{a}_m^\top \boldsymbol{x} + r_m \bigr) \bigr) \boldsymbol{a}_m,
\end{align}
where $\sigma'$ denotes the derivative of the activation function $\sigma$ (e.g., if $\sigma = \tanh(x)$, then $\sigma' = \text{sech}^2(x)$). It should be noted that at the initial stage, we directly adopt TransNet to obtain the initial‑solution information. This step pre‑generates approximate gradient information for the solution, which does not need to be sufficiently accurate. Therefore, the number of neurons at the initial step $M_0$ is typically chosen as a small value to reduce the computational cost in practice.

\begin{rmk}    
	In a specific iteration, since the output-layer neuron parameters $\alpha_m$ need to be determined by minimizing the loss function \eqref{loss function}, the neuron derivatives $\gamma_m \sigma' \bigl( \gamma_m \bigl( \boldsymbol{a}_m^\top \boldsymbol{x} + r_m \bigr) \bigr) \boldsymbol{a}_m$ are usually already available, so the approximate derivatives \eqref{approximate derivatives} can be easily obtained with essentially no additional computational cost.
\end{rmk}

We employ an iterative neuron expansion strategy to acquire approximate gradient information to guide neuron construction. The key to the neuron expansion lies in retaining the neurons used in previous iterations, i.e., in the $i$-th iteration, only $M_i - M_{i-1}$ new neurons are resampled using the methods described in Sections \ref{location parameters} and \ref{shape parameters}. It is particularly important to note that when selecting shape parameters, those for the newly added $M_i - M_{i-1}$ neurons should be calculated using the position parameters of all $M_i$ neurons via the generation formula \eqref{shape parameters formula}, rather than only from those of the $M_i - M_{i-1}$ new neurons. The continuous projection onto the neural feature space, obtained iteratively via the neuron expansion method, exhibits monotonicity, which implies that the method possesses certain iterative stability, as will be elaborated on in Section~\ref{4.1}. In summary, the complete algorithm of the proposed A-TransNet is presented in Algorithm \ref{A-TransNet algorithm}.

\vspace{-0.2cm}
\begin{algorithm}[!htbp]
	\caption{The A-TransNet Method}
	\label{A-TransNet algorithm}
	\KwIn{Interior, boundary and test points $\{\bx_j\}_{j=1}^{J_I}$, $\{\bx_j\}_{j=1}^{J_B}$ and $\{\bx_j\}_{j=1}^{J_T}$, number of hidden layer neurons at each iteration $\{M_s\}_{s=0}^{S}$, radius of domain $R$, allocation parameter $\mu$, empirical constant $C$, search neighborhood radius $\delta$, tolerance $\epsilon$.}
	\KwOut{Numerical solution $u_{\rm NN}(\bx)$.}
	
	Compute the initial numerical solution $u_{\mathrm{NN}}^{(0)}$ with an
	$M_0$-scale TransNet, and compute the approximate gradient information
	$\{\boldsymbol{g}_i\}_{i=1}^{J_I}$ of the solution via
	\eqref{approximate derivatives}.
	
	\For{$s = 1:S$}{
		Let $\ell_s=1$ for $s=1$, and $\ell_s=M_{s-1}+1$ for $s>1$.
		
		Call Algorithm \ref{location parameters algorithm} to construct location parameters $\{\boldsymbol{a}_m\}_{m=\ell_s}^{M_s}$ and $\{r_m\}_{m=\ell_s}^{M_s}$ for the newly added neurons, with inputs $R$, $M_s-\ell_s+1$, $\mu$, $\{\boldsymbol{g}_i\}_{i=1}^{J_I}$, and $\{\bx_j\}_{j=1}^{J_I}$.
		
		Call Algorithm~\ref{shape parameters algorithm} to generate
		shape parameters $\{\gamma_m\}_{m=1}^{M_s}$, and use only $\{\gamma_m\}_{m=\ell_s}^{M_s}$ for the newly added neurons, with inputs $\{\boldsymbol{a}_m\}_{m=1}^{M_s}$,
		$\{r_m\}_{m=1}^{M_s}$, $R$, $\delta$, and $C$.
		
		Minimize the loss function \eqref{loss function} via the least squares solver to obtain the output layer neuron parameters $\{\alpha_m\}_{m=0}^{M_s}$.
		
		Compute the iterative solution $u^{(s)}_{\rm NN}(\bx)$ via \eqref{rewritten solution} and update the approximate gradient information $\{\boldsymbol{g}_i\}_{i=1}^{J_I}$ via \eqref{approximate derivatives}.
		
		\textbf{if } $\left\|u_{\rm NN}^{(s)}(\bx) - u_{\rm NN}^{(s-1)}(\bx)\right\| < \epsilon$ or s=S, \textbf{then}
		$u_{\rm NN}(\bx) \leftarrow u_{\rm NN}^{(s)}(\bx)$ and
		\textbf{break} 
	}
	
	Return numerical solution $u_{\rm NN}(\bx)$.
\end{algorithm}

\begin{rmk}
	\vspace{-0.2cm}
	The neuron expansion method requires increasing the number of neurons during iteration, which enlarges the neural feature space of the current iteration compared with the previous one, and further exhibits monotonicity and convergence properties. With consideration of computational efficiency and numerical accuracy, following the settings in \cite{Ananth2021GNN}, we recommend increasing the number of neurons in a power-of-two exponential manner, i.e., $M_i = M_1 \times 2^{i-1}$, which is adopted in the subsequent numerical experiments.
	\vspace{-0.1cm}
\end{rmk}

\section{Numerical Analysis}\label{sec:analysis}

In this section, we conduct the numerical analysis of A-TransNet. Section~\ref{4.1} presents and analyzes the properties of continuous projections, while Section~\ref{4.2} investigates the convergence and a posteriori error estimates for numerical solutions.

\subsection{Analysis for Continuous Projections} \label{4.1}

We first consider the general form \eqref{PDE} of the PDE and define the specific norm accordingly as follows:
\begin{align*}
	\| u \|_{I,B} = \left( \|\lambda_I \mathcal{L} u\|_{L^2(\Omega)}^2 + \|\lambda_B \mathcal{B} u\|_{L^2(\partial\Omega)}^2 \right)^{\frac{1}{2}},
\end{align*}
where $\lambda_I, \lambda_B \ge 0$ are non-negative penalty parameters that balance the weights of the interior linear operator $\mathcal{L}u$ and the boundary linear operator $\mathcal{B}u$ in the norm.

At the $s$-th iterative step, we project the exact solution $u$ onto the subspace $\mathcal{P}_\text{NN}^{(s)}$ spanned by the neural basis functions as
\begin{align} \label{continuous projection minimizer}
	u_s^* = \mathop{\arg\min}\limits_{v \in \mathcal{P}_\text{NN}^{(s)}} \|u - v\|_{I,B},
\end{align}
where $P$ denotes the continuous projection operator. Next, the error is defined as $e_s^* = u - u_s^*$. Due to the properties of the continuous projection and the expansion of the spanned subspace $\mathcal{P}_\text{NN}$ during iterations in the neuron expansion method, we have $\mathcal{P}_\text{NN}^{(s-1)} \subset \mathcal{P}_\text{NN}^{(s)}$. It is clear that the norm of error $\|e_s^*\|_{I,B}$ is non-increasing with iterations, i.e., $\| e_{s+1}^* \|_{I,B} \leq \| e_s^* \|_{I,B}$. Proposition~\ref{monotonicity of iterative error} rigorously establishes the monotonicity of the iterative error for continuous projections.

\begin{proposition}[Monotonicity of Projection Error in Neuron Expansion Method] \label{monotonicity of iterative error}
	Consider the general form of the PDE \eqref{PDE} with a source term $f\in L^2(\Omega)$ and boundary data $g\in L^2(\partial\Omega)$, where the differential operator $\mathcal{L}$ and boundary operator $\mathcal{B}$ are linear. By iteratively expanding the subspace $\mathcal{P}_\text{NN}$ spanned by neural basis functions, the norm of the error is non-increasing, i.e., $\| e_{s+1}^* \|_{I,B} \leq \| e_s^* \|_{I,B}$. In particular, equality holds if and only if $u_{s+1}^* = u_s^*$. 
\end{proposition}

By the universal approximation theorem \cite{Hornik1991Approx}, for any $\tau_s > 0$ in the $s$-th iteration, there exists a network size $n(\tau_s, u^*_{s-1})$ such that whenever $n \ge n(\tau_s, u^*_{s-1})$, one can find $\psi_s \in V_n^\sigma$ (where $V_n^\sigma$ denotes the function space of neural networks with $n$ neurons and activation $\sigma$) satisfying
\begin{align} \label{assumption}
	\frac{\| e_{s-1}^* - \psi_s \|_{I,B}}{\| e_{s-1}^* \|_{I,B}} \leq \tau_s.
\end{align}
In this case, a posteriori error estimates and convergence analysis are further established in Proposition~\ref{posteriori error estimation and convergence}, which is rigorously proved in \cite{Weng2026DCM}.

\begin{proposition}[A Posteriori Error Estimates and Convergence of Projection] \label{posteriori error estimation and convergence}
	Define the variation of the solution between successive iterative steps as $v_s^* = u_s^* - u_{s-1}^*$. Let $S \ge 1$. At stage $s > 0$, given $0 < \tau_s < 1$, if \eqref{assumption} holds for all $s = 1, \dots, S$, then
	\begin{align} \label{a posteriori error estimation}
		\frac{1}{1+\tau_s}\| v_s^* \|_{I,B} \leq \| u - u_{s-1}^* \|_{I,B} \leq \frac{1}{1-\tau_s}\| v_s^* \|_{I,B}
	\end{align}
	and
	\begin{align} \label{iterative convergence of the neural basis expansion method}
		\| u - u_S^* \|_{I,B} \leq \| u - u_0^* \|_{I,B} \cdot \prod_{s=1}^S \min\{1, 2\tau_s\}.
	\end{align}
\end{proposition}

\begin{rmk}    
	Proposition~\ref{monotonicity of iterative error} demonstrates the monotonicity of the continuous projection error in the neural expansion method, which further implies its stability. Proposition~\ref{posteriori error estimation and convergence} further provides the convergence and a posteriori error estimates for continuous projections. Note that this section only discusses properties of continuous projections. Therefore, in Subsection~\ref{4.2}, we further analyze the convergence and error estimates for numerical solutions.
\end{rmk}

\subsection{Analysis for Numerical Solutions} \label{4.2}

In this subsection, we further conduct an error analysis of A-TransNet and thereby establish its convergence properties and a posteriori error estimates for numerical solutions. Let the differential operator $\mathcal{L}$ and the boundary operator $\mathcal{B}$ in \eqref{PDE} be bounded linear operators, and let $\Omega \subset \mathbb{R}^d$ be a bounded Lipschitz domain.

The precise loss function is defined as
\begin{align} \label{continuous loss}
	{Loss}(u) := {\lambda_I}^2 \|\mathcal{L}u - f\|^2_{L^2(\Omega)} +  {\lambda_B}^2 \|\mathcal{B}u - g\|^2_{L^2(\partial\Omega)},
\end{align}
which can be discretized and approximated by $\hat{Loss}(u)$ \eqref{loss function}. Thus, the best approximation of the exact solution $u$ at the $s$-th iterative step in the neural feature space $\mathcal{P}_\text{NN}^{(s)}$ is given by \eqref{continuous projection minimizer}. The numerical solution obtained by minimizing the approximate loss function \eqref{loss function} via the least squares method is
\begin{align*}
	u^{(s)}_\rho := \underset{u \in \mathcal{P}^{(s)}_\text{NN}}{\arg\min} \hat{Loss}(u).
\end{align*}
We state the following two key assumptions on the boundedness of the adaptive parameters and the admissibility condition for the activation functions.

\begin{assumption} [Bounded Adaptive Parameters of A-TransNet] \label{bounded adaptive parameters of A-TransNet}
	All inner parameters $(\bm{a}_m, \gamma_m, r_m)$ of A-TransNet
	are sampled from a fixed compact set
	\begin{align*}
		\Theta = S^{d-1} \times [\gamma_{\min}, \gamma_{\max}] \times [-R, R]
		\subset \mathbb{R}^{d+2},
	\end{align*}
	where $\bm{a}_m \in S^{d-1}$ (unit direction),
	$\gamma_{\min} \le \gamma_m \le \gamma_{\max}$ (bounded shape),
	$|r_m| \le R$ (bounded translation). And the coefficients satisfy
	\begin{align*}
	\sum_{m=1}^{M_s} |c_m| \le C_N,
	\end{align*}
	where $C_N>0$ is a constant independent of the number of collocation points.
\end{assumption}

\begin{assumption}[Admissibility Condition]
	\label{assump:activation_admiss}
	The activation function $\sigma$ satisfies:
	\begin{enumerate}[(1)]
		\item $\sigma \in C^s(\mathbb{R})$, and $\|\sigma^{(s_0)}\|_{L^\infty(\mathbb{R})} < \infty$ for all $s_0 = 0,1,\dots,s$, where $C^s(\mathbb{R})$ denotes the space of functions with continuous derivatives up to order $s$ on $\mathbb{R}$.
		\item $\widehat{\sigma}'(1) \neq 0$ and $\sigma^{(s_0)} \in L^1(\mathbb{R})$ for all $s_0 = 1,\dots,s+1$, where $\widehat{\sigma}$ represents the fourier transform of $\sigma$.
	\end{enumerate}
\end{assumption}

\begin{rmk}    
	Assumption~\ref{bounded adaptive parameters of A-TransNet} is inherently justified by the parameter design of A-TransNet, where the parameter $\bm{a}_m$ is normalized to the unit sphere, and $\gamma_m$ and $r_m$ are naturally bounded by the finite number of hidden layer neurons $M$ and the bounded computational domain $\Omega$. The function class in Assumption~\ref{assump:activation_admiss} covers $tanh$, logistic, and other smooth sigmoid functions.
\end{rmk}

\begin{lma}[Statistical Component Bound]\label{loss gap}
	Suppose Assumptions \ref{bounded adaptive parameters of A-TransNet} and \ref{assump:activation_admiss} hold. Let $\mathrm{Loss}(u)$ and $\hat{\mathrm{Loss}}(u)$ be the continuous weighted loss and discrete empirical loss defined in \eqref{continuous loss} and \eqref{loss function}. Let $C_R>0$ denote a uniform bound such that $|\mathcal{L}u(\bx)|, |f(\bx)| \le C_R$ for all $\bx\in\Omega$, $|\mathcal{B}u(\bx)|, |g(\bx)| \le C_R$ for all $\bx\in\partial\Omega$, and $\lambda_{max} = \max\{\lambda_I,\lambda_B\}$. For fixed $s$, with probability at least $1-\delta_s$, there exists
	\begin{align*}
	\sup_{u\in\mathcal P^{(s)}_{\mathrm{NN}}}\big|{Loss}(u)-\hat{{Loss}}(u)\big| \le
	C_N^2 \lambda_{max}^2 C_R^2 M_s \left( \sqrt{\frac{\log(6M_s/\delta_s)}{J_I}} + \sqrt{\frac{\log(6M_s/\delta_s)}{J_B}} \right) = \hat{\varepsilon}_s.
	\end{align*}
\end{lma}

The original theorem concerning the unit-weight residual loss ($\lambda_I=\lambda_B=1$) is rigorously stated and proved in Theorem 4.14 of \cite{Dang2024AGRaNN}. For our weighted loss with penalty parameters $\lambda_I,\lambda_B>0$, the asymptotic scaling with respect to the number of collocation points remains unchanged, and only the constant prefactor is scaled by $\lambda_I^2$ and $\lambda_B^2$. The empirical error bound is given in Theorem~\ref{empirical error bound}.

\begin{thm} [Empirical Error Bound]\label{empirical error bound}
	Under the conditions of Lemma \ref{loss gap}, let $\varepsilon_s =\sqrt{2\hat{\varepsilon}_s}$. Then
	\begin{align*}
	\|u_\rho^{(s)}-u_s^*\|_{I,B} \le \varepsilon_s.
	\end{align*}
\end{thm}

\begin{proof}
	Let $e_s = u_\rho^{(s)} - u_s^*$, so that $e_s\in \mathcal{P}^{(s)}_\text{NN}$. By empirical minimality, $\hat{Loss}(u_\rho^{(s)}) \le \hat{Loss}(u_s^*)$. We perform the four-term decomposition:
	\begin{align*}
	\|u_\rho^{(s)} - u\|_{I,B}^2
	=& \Big(\|u_\rho^{(s)} - u\|_{I,B}^2 - \hat{Loss}(u_\rho^{(s)})\Big)
	+ \Big(\hat{Loss}(u_\rho^{(s)}) - \hat{Loss}(u_s^*)\Big) \\
	+& \Big(\hat{Loss}(u_s^*) - \|u_s^* - u\|_{I,B}^2\Big)
	+ \|u_s^* - u\|_{I,B}^2.
	\end{align*}
	Since $\hat{Loss}(u_\rho^{(s)}) - \hat{Loss}(u_s^*) \le 0$, ${Loss}(u_\rho^{(s)}) = \|u_\rho^{(s)} - u\|_{I,B}^2$ and $ {Loss}(u_s^*) = \|u_s^* - u\|_{I,B}^2$, we obtain
	\begin{align*}
	\|u_\rho^{(s)} - u\|_{I,B}^2 \le \|u_s^* - u\|_{I,B}^2 + 2 \sup_{u\in\mathcal P^{(s)}_{\mathrm{NN}}}\big|\mathrm{Loss}(u)-\hat{\mathrm{Loss}}(u)\big|.
	\end{align*}
	Recall that $u_s^*$ is the orthogonal projection of $u$ onto $\mathcal{P}^{(s)}_\text{NN}$ with respect to the inner product induced by $\|\cdot\|_{I,B}$. Since $e_s\in \mathcal{P}^{(s)}_\text{NN}$, the orthogonality property yields
	\begin{align*}
	\|u_\rho^{(s)} - u\|_{I,B}^2
	= \big\|(u_s^* - u) + e_s\big\|_{I,B}^2
	= \|u_s^* - u\|_{I,B}^2 + \|e_s\|_{I,B}^2.
	\end{align*}
	Canceling $\|u_s^* - u\|_{I,B}^2$ from both sides, we apply Lemma~\ref{loss gap} to obtain
	\begin{align*}
	\|e_s\|_{I,B}^2 \le 2\hat{\varepsilon}_s.
	\end{align*}
	Taking the square root of both sides gives
	\begin{align*}
	\|u_\rho^{(s)} - u_s^*\|_{I,B} \le \sqrt{2\hat{\varepsilon}_s} = \varepsilon_s.
	\end{align*}
	This completes the proof.
\end{proof}

From Proposition~\ref{posteriori error estimation and convergence} and Theorem~\ref{empirical error bound}, the convergence and a posteriori error estimates of A-TransNet are obtained and presented in Corollary~\ref{convergence of A-TransNet} and Corollary~\ref{a posteriori error estimates of A-TransNet}, respectively.

\begin{corollary}[Convergence of A-TransNet] \label{convergence of A-TransNet}
	Assume that the assumptions of Proposition \ref{posteriori error estimation and convergence} and Lemma \ref{loss gap} hold. Let $\varepsilon = \varepsilon_0\prod_{s=1}^S \min\{1,2\tau_s\} + \varepsilon_S$. Then the total error of empirical minimizer $u_\rho^{(S)}$ obeys
	\begin{align*}
	\|u - u_\rho^{(S)}\|_{I,B} \le \|u - u_\rho^{(0)}\|_{I,B}\prod_{s=1}^S \min\{1,2\tau_s\} + \varepsilon.
	\end{align*}
\end{corollary}

\begin{proof}
	By the triangle inequality for $\|\cdot\|_{I,B}$ and Theorem \ref{empirical error bound}, it follows that
	\begin{align*}
		\|u-u_\rho^{(S)}\|_{I,B}
		\le
		\|u_\rho^{(S)}-u_S^*\|_{I,B}+\|u_S^*-u\|_{I,B}
		\le
		\|u-u_S^*\|_{I,B} + \varepsilon_S, \\
		\|u-u_0^*\|_{I,B}
		\le
		\|u-u_\rho^{(0)}\|_{I,B}+\|u_\rho^{(0)}-u_0^*\|_{I,B}
		\le
		\|u-u_\rho^{(0)}\|_{I,B} + \varepsilon_0.
	\end{align*}
	Substituting \eqref{iterative convergence of the neural basis expansion method} yields
	\begin{align*}
		\|u - u_\rho^{(S)}\|_{I,B} \le \|u - u_\rho^{(0)}\|_{I,B}\prod_{s=1}^S \min\{1,2\tau_s\} + \varepsilon_0\prod_{s=1}^S \min\{1,2\tau_s\} + \varepsilon_S.
	\end{align*}
	This completes the proof.
\end{proof}

\begin{rmk}
	The convergence result in Corollary~\ref{convergence of A-TransNet} is of a saturation-type nature. The term $\varepsilon$ represents the saturation plateau, which arises from the sampling error of empirical minimization. As the number of collocation points increases, this saturation plateau can be made arbitrarily small, and the bound then reduces to a contraction estimate governed by the product $\prod_{s=1}^S \min\{1,2\tau_s\}$.
\end{rmk}

\begin{corollary}[A Posteriori Error Estimates of A-TransNet] \label{a posteriori error estimates of A-TransNet}
	Assume that the assumptions of Proposition \ref{posteriori error estimation and convergence} and Lemma \ref{loss gap} hold. Let
	$\eta_s^- := \frac{1}{1+\tau_s}\big(\varepsilon_s +(2+\tau_s)\varepsilon_{s-1}\big)$ and $\eta_s^+ := \frac{1}{1-\tau_s}\big(\varepsilon_s +(2-\tau_s)\varepsilon_{s-1}\big)$. Then 
	\begin{align*}
	\frac{1}{1+\tau_s}\|v_\rho^{(s)}\|_{I,B} - \eta_s^-
	\le \|u-u_\rho^{(s-1)}\|_{I,B}
	\le \frac{1}{1-\tau_s}\|v_\rho^{(s)}\|_{I,B} + \eta_s^+.
	\end{align*}
\end{corollary}

\begin{proof}
	By the triangle inequality for $\|\cdot\|_{I,B}$ together with Theorem~\ref{empirical error bound}, we have
	\begin{align*}
		\|u-u_\rho^{(s-1)}\|_{I,B}
		\le
		\|u-u_{s-1}^*\|_{I,B} + \|u_\rho^{(s-1)}-u_{s-1}^*\|_{I,B}
		\le
		\|u-u_{s-1}^*\|_{I,B}+\varepsilon_{s-1}, \\
		\|u-u_{s-1}^*\|_{I,B}
		\le
		\|u-u_\rho^{(s-1)}\|_{I,B} + \|u_\rho^{(s-1)}-u_{s-1}^*\|_{I,B}
		\le
		\|u-u_\rho^{(s-1)}\|_{I,B}+\varepsilon_{s-1}.
	\end{align*}
	Thus
	\begin{align*}
		\|u-u_{s-1}^*\|_{I,B}-\varepsilon_{s-1}
		\le\|u-u_\rho^{(s-1)}\|_{I,B}
		\le\|u-u_{s-1}^*\|_{I,B}+\varepsilon_{s-1}.
	\end{align*}
	
	Since $v_s^* = u_s^* - u_{s-1}^*$, we derive a bound for its norm via the triangle inequality and Theorem~\ref{empirical error bound}:
	\begin{align*}
		\|v_s^*\|_{I,B}
		&=\|u_s^*-u_{s-1}^*\|_{I,B}\\
		&\le\|u_s^*-u_\rho^{(s)}\|_{I,B}
		+\|u_\rho^{(s)}-u_\rho^{(s-1)}\|_{I,B}
		+\|u_\rho^{(s-1)}-u_{s-1}^*\|_{I,B}\\
		&\le\|v_\rho^{(s)}\|_{I,B}+\varepsilon_s+\varepsilon_{s-1}.
	\end{align*}
	Similarly, starting from $\|v_\rho^{(s)}\|_{I,B}$, we obtain
	\begin{align*}
		\|v_\rho^{(s)}\|_{I,B}
		\le\|v_s^*\|_{I,B}+\varepsilon_s+\varepsilon_{s-1}.
	\end{align*}
	Combining the above with \eqref{a posteriori error estimation}, we arrive at
	\begin{align*}
		\|u-u_\rho^{(s-1)}\|_{I,B} \ge \frac{1}{1+\tau_s}\|v_\rho^{(s)}\|_{I,B} - \frac{1}{1+\tau_s}(\varepsilon_s +(2+\tau_s)\varepsilon_{s-1}), \\
		\|u-u_\rho^{(s-1)}\|_{I,B} \le \frac{1}{1-\tau_s}\|v_\rho^{(s)}\|_{I,B} + \frac{1}{1-\tau_s}(\varepsilon_s +(2-\tau_s)\varepsilon_{s-1}).
	\end{align*}
	This completes the proof.
\end{proof}

\begin{rmk}
	Due to the presence of the explicit perturbation terms $\eta_s^-$ and $\eta_s^+$, the two-sided a posteriori error bound derived in Corollary~\ref{a posteriori error estimates of A-TransNet} does not constitute a rigorous deterministic stopping condition. Although the stopping rule adopted in Algorithm~\ref{A-TransNet algorithm} is motivated by this bilateral error relation, it remains a heuristic criterion. Nevertheless, this bound inspires our a posteriori error estimation strategy and holds with high probability for most practical cases.
\end{rmk}

\section{Numerical Experiments} \label{Sec:NumerExper}

In this section, we apply the proposed A-TransNet to various types of partial differential equations, thus validating the effectiveness of our algorithm.

All experiments are conducted in a Python environment with the following configurations. When selecting the shape parameters of hidden layer neurons, the number of points $N_m$ in the neighborhood defined in \eqref{shape parameters formula} is counted using the KD-Tree method. For solving the least squares problem, the SVD method is adopted, and the $\texttt{cp.linalg.lstsq}$ function from the CuPy library is invoked with parameter $\texttt{rcond = None}$ to accelerate the computation. All experiments are conducted on a workstation with an NVIDIA A100 GPU. The random seed is set to 42 for reproducibility. Moreover, to mitigate random effects, results are averaged across 10 independent trials.

We uniformly adopt the hyperbolic tangent function $\text{tanh}(x)=\frac{e^x-e^{-x}}{e^x+e^{-x}}$ as the activation function $\sigma(x)$. The accuracy in the numerical experiments is measured by the $L^2$ and $L^\infty$ norms, both of which can be approximately calculated via the grid method as
\begin{align*}
	E_2 = \|u - \hat{u}\|_{L^2} \approx \sqrt{\frac{1}{J_T}\sum_{i=1}^{J_T} |u(x_i) - \hat{u}(x_i)|^2}, \quad E_\infty = \|u - \hat{u}\|_{L^\infty} \approx \max_{1\le i\le J_T} |u(x_i) - \hat{u}(x_i)|,
\end{align*}
where $J_T$ denotes the number of test points.

The main parameters adopted in our experiments are summarized as follows:
\begin{itemize}
	\item $J_I$, $J_B$ and $J_T$: numbers of interior, boundary and test points, respectively.
	
	\item $M_s$: number of hidden layer neurons at the $s$-th iteration step in A-TransNet. Particularly, $M_0$ denotes the number of hidden layer neurons for the small-scale TransNet pretrained to approximate gradients.
	
	\item $R$: radius for sampling the location parameters of neurons.
	
	\item $\mu$: allocation parameter in $[0,1]$ controlling the sampling strategy of neuron location parameters. $\mu M$ neurons are sampled via the active subspaces method to capture solution anisotropy, while the remaining $(1-\mu)M$ neurons are sampled via the local gradient method to capture solution sharpness.
	
	\item $C$: empirical constant that depends on the specific form and parameters of the PDE, and is typically obtained via the small-scale golden search method.
\end{itemize}
The detailed parameters in the numerical experiments are listed in Table \ref{table1}.

\begin{table}[!htbp]
	\centering
	\vspace{-0.2cm}
	\begin{minipage}{0.815\textwidth}
		\centering
		\caption{Detailed parameters used in the numerical experiments.}
		\resizebox{\linewidth}{!}{
			\begin{tabular}{cccccccc}
				\hline
				Examples & $J_I$ & $J_B$ & $J_T$ & $M_s$ & $R$  & $\mu$ & $C$ \\ \hline
				Sec \ref{Subsec:5.1} $\varepsilon = 10^{-2}$ & 700 & 2 & 1401 & $20 \times 2^{s}$ & 1.2 & $2/3$ & 0.10 \\
				Sec \ref{Subsec:5.1} $\varepsilon = 10^{-3}$& 5000 &  2 & 10001 & $50 \times 2^{s}$ & 1.2 & $2/3$ & 0.10 \\
				Sec \ref{Subsec:5.2} Case 1& $50 \times 50$ & 200 & $101 \times 101$ & $50 \times 2^{s}$ & 1.5 & $2/3$ & 0.08 \\
				Sec \ref{Subsec:5.2} Case 2 & $50 \times 50$ & 200 & $101 \times 101$ & $50 \times 2^{s}$ & 1.5 & $2/3$ & 0.10 \\
				Sec \ref{Subsec:5.2} Case 3& $50 \times 50$ & 200 & $101 \times 101$ & $50 \times 2^{s}$ & 1.5 & $2/3$ & 0.08 \\
				Sec \ref{Subsec:5.3} & $150 \times 150$ & 600 & $301 \times 301$ & $500 \times 2^{s}$ & 1.5 & $1/3$ & 0.14 \\
				Sec \ref{Subsec:5.4} Case 1 & $150 \times 150$ & 600 & $301 \times 301$ & $500 \times 2^{s}$ & 1.5 & $2/3$ & 0.11 \\
				Sec \ref{Subsec:5.4} Case 2 & $150 \times 150$ & 600 & $301 \times 301$ & $500 \times 2^{s}$ & 1.5 & $2/3$ & 0.11 \\ 
				Sec \ref{Subsec:5.5}& $50 \times 50$ & 200 & $101 \times 101$ & $150 \times 2^{s}$ & 1.5 & $2/3$ & 0.08\\
				Sec \ref{Subsec:5.6} Case 1& $100 \times 100$ & 400 & $201 \times 201$ & $1000\times 2^{\max(s-1,\,0)}$ & 1.5 & $2/3$ & 0.10 \\
				Sec \ref{Subsec:5.6} Case 2& $40 \times 40 \times 40 $ & 480 & $81 \times 81 \times 81$ & $1500\times 2^{\max(s-1,\,0)}$ & 1.8 & $2/3$ & 0.14 \\
				Sec \ref{Subsec:5.7} & $150 \times 300$ & 1800 & $251 \times 512$ & $500\times 2^s$ & 1.5 & $2/3$ & 0.18 \\ \hline
			\end{tabular}
		}
		\label{table1}
	\end{minipage}
	\vspace{-0.2cm}
\end{table}

The competing models are presented as follows. TransNet refers to the method presented in Subsection~\ref{TransNet}. PINN denotes the standard physics-informed neural network \cite{Raissi2019PINNs}, where we adopt a four-layer architecture with $M/4$ neurons per layer to ensure a fair comparison. AG-RaNN (single hidden layer, SHL) is the single-hidden-layer form of the method proposed in \cite{Dang2024AGRaNN}.

\subsection{One-Dimensional Convection-Diffusion Equation with Boundary Layers} \label{Subsec:5.1}

Here we consider the one-dimensional convection-diffusion equation with boundary layers as
\begin{align*}
	\begin{aligned}
		-\varepsilon u''(x) - u'(x) &= 1, \quad x \in \Omega, \\
		u(x) &= 0, \quad x \in \partial \Omega,
	\end{aligned}
\end{align*}
where $\Omega = (-1, 1)$. The solution to this problem comprises a smooth component together with a boundary layer term $e^{\frac{1-x}{\varepsilon}}$:
\begin{align*}
	u(x) = -x -
	\left(
	\frac{e^{\frac{2}{\varepsilon}} + 1}{e^{\frac{2}{\varepsilon}} - 1}
	\right)
	\left(
	\frac{2e^{\frac{1-x}{\varepsilon}}}{e^{\frac{2}{\varepsilon}} + 1} - 1
	\right),
\end{align*}
where $\varepsilon$ is a small positive parameter that determines the difficulty of solving this problem. In the following numerical tests, $\varepsilon = 10^{-2}$ and $\varepsilon = 10^{-3}$ are adopted, with the test problem taken from \cite{Weng2026DCM}. The small parameter induces sharp variations of the exact solution near $x=-1$, where the gradient can reach nearly $2000$, forming a boundary layer. This feature poses substantial challenges for obtaining high-precision numerical solutions \cite{Weng2026DCM}.

In this subsection, the tolerance $\epsilon$ of A‑TransNet is set to $10^{-6}$ to verify the effectiveness of a posteriori error estimates for guiding iteration termination. Given the substantial difficulty variation across different problems and the fact that A‑TransNet generally requires few iterations (typically within 10 steps), subsequent examples terminate iterations only by the prescribed number of iterations without utilizing the tolerance criterion.

The numerical results of A-TransNet are shown in Figure \ref{fig:5.1}, where the left and right panels correspond to the cases $\varepsilon=10^{-2}$ and $\varepsilon=10^{-3}$, respectively. It can be observed that the problem becomes more challenging for $\varepsilon=10^{-3}$ compared with $\varepsilon=10^{-2}$, requiring more neurons to achieve comparable accuracy. The subplots on the left of each panel present the exact solution, the numerical solution obtained by A-TransNet, and the pointwise error, demonstrating that A-TransNet accurately captures the sharp boundary layer near $x=-1$. The log-scale error curves with respect to iteration number $s$ on the right indicate that both $E_2$ and $E_\infty$ decay rapidly within a small number of iterations, consistent with the convergence properties established in our theoretical analysis. 

Furthermore, a detailed comparison between these results and those obtained by the standard TransNet, PINN, and AG-RaNN (SHL) is presented in Table \ref{table5.1}. For $\varepsilon = 10^{-2}$, A-TransNet achieves an error of $L_2$ norm of 4.50e-09 with only 160 neurons in 0.05 seconds, while the deep collocation method proposed in \cite{Weng2026DCM} requires 640 neurons to attain an accuracy on the order of $10^{-6}$. In particular, it can be observed that capturing the rapid variation of the solution near $x=-1$ is extremely challenging when $\varepsilon = 10^{-3}$. The best-performing competing model AG-RaNN (SHL) achieves an error of 7.56e-05 in $0.79$ seconds. By contrast, A-TransNet attains an error of $L_2$ norm on the order of $10^{-7}$ with merely 800 neurons within $0.72$ seconds, demonstrating its superior performance.

\begin{figure}[!htb]
	\centering
	\vspace{-0.2cm}
	\includegraphics[width=0.98\textwidth]{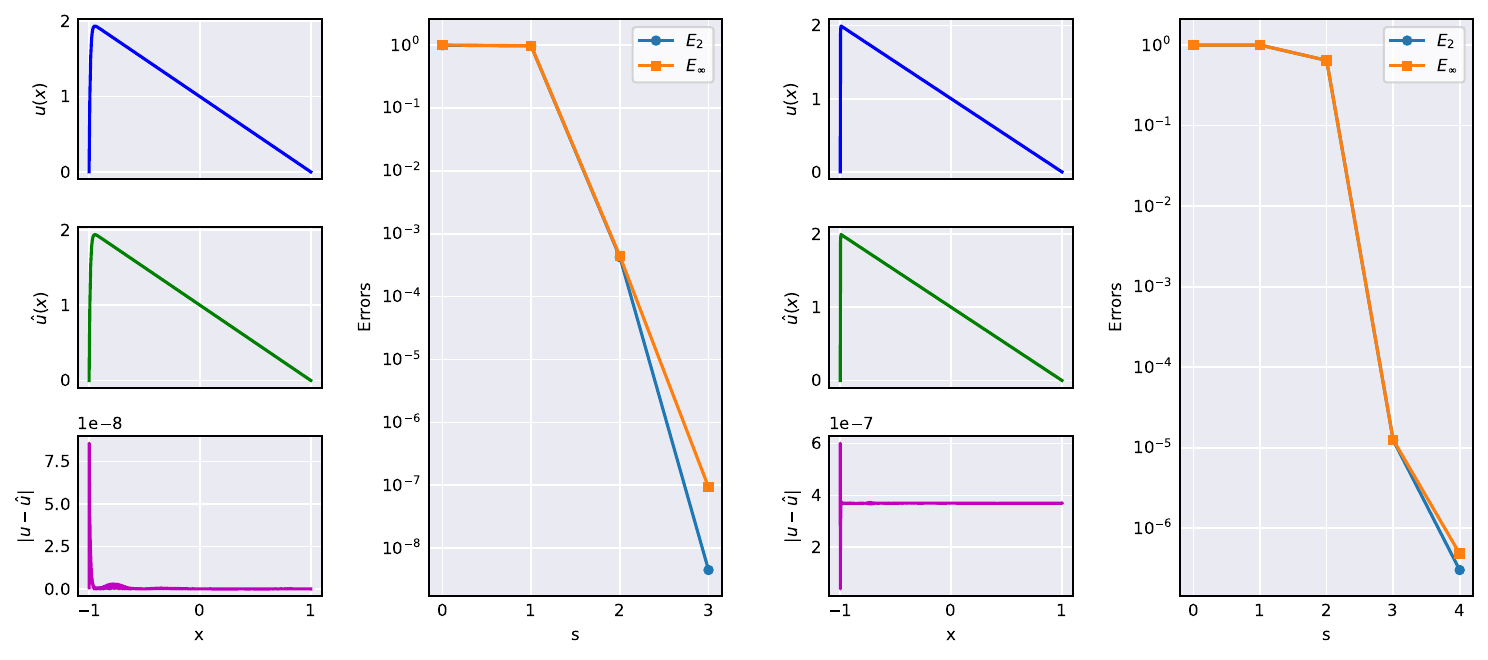}
	\vspace{-0.35cm}
	\caption{Numerical results of A-TransNet for $\varepsilon = 10^{-2}$ (the left two columns) and $\varepsilon = 10^{-3}$ (the left two columns) in 5.1.}
	\label{fig:5.1}
\end{figure}
\vspace{-0.6cm}
\begin{table}[!htbp]
	\centering
	\vspace{-0.2cm}
	\begin{minipage}{0.7\textwidth}
		\centering
		\caption{Comparison of errors for examples in 5.1.}
		\resizebox{\linewidth}{!}{
			\begin{tabular}{ccccccc}
				\hline
				\multirow{3}{*}{Methods} & \multicolumn{3}{c}{$\varepsilon=10^{-2}$}& \multicolumn{3}{c}{$\varepsilon=10^{-3}$} \\ \cmidrule(lr){2-4}\cmidrule(lr){5-7} & \multicolumn{3}{c}{$M=160$}&\multicolumn{3}{c}{$M=800$} \\ \cmidrule(lr){2-4}\cmidrule(lr){5-7}
				& $E_2$ & $E_\infty$ & Time(s) & $E_2$ & $E_\infty$ & Time(s)  \\ \hline
				TransNet & 1.42e-04 & 1.48e-04 & 0.03 & 7.24e-03 & 7.25e-03 & 0.62 \\
				PINN & 7.98e-05 & 1.80e-04 & 96.52 & 1.04e-00 & 1.50e-00 & 97.18  \\
				AG-RaNN (SHL) & 5.53e-06 & 7.67e-06 & 0.05 & 7.56e-05 & 7.63e-05 & 0.79  \\ 
				\textbf{A-TransNet} & \textcolor{red}{4.50e-09} & \textcolor{red}{9.35e-08} & \textcolor{red}{0.05} & \textcolor{red}{3.03e-07} & \textcolor{red}{4.84e-07} & \textcolor{red}{0.72}  \\\hline
			\end{tabular}
		}
		\label{table5.1}
	\end{minipage}
	\vspace{-0.3cm}
\end{table}

\subsection{Poisson Equation with Anisotropic Solution} \label{Subsec:5.2}

Consider the Poisson equation given by
\begin{align} \label{poisson}
	\begin{aligned}
		-\Delta u(\bx) &= f(\bx), \quad \bx \in \Omega, \\
		u(\bx) &= g(\bx), \quad \bx \in \partial \Omega,
	\end{aligned}
\end{align}
where $\Omega = (-1, 1)^2$. We investigate three test cases with distinct anisotropic analytical solutions adopted from \cite{Pieper2024Nonuniform}.

\begin{description}
	\item[Case 1:] 
	\begin{align*}
		u(\bx) = \sin\left(5x_1 + 5\sqrt{2}x_2\right), \quad f(\bx) = 75\sin\left(5x_1 + 5\sqrt{2}x_2\right).
	\end{align*}
	\item[Case 2:] 
	\begin{align*}
		u(\bx) = e^{-(x_1 + 2x_2)^2}, \quad f(\bx) = -\left(20x_1^2 + 80x_1x_2 + 80x_2^2 - 10\right)e^{-(x_1 + 2x_2)^2}.
	\end{align*}
	\item[Case 3:] 
	\begin{align*}
		u(\bx) = 10\left(x_1 + 0.5x_2 + 3\right)^{-3}, \quad f(\bx) = -150\left(x_1 + 0.5x_2 + 3\right)^{-5}.
	\end{align*}
\end{description}

Figure \ref{fig:5.2.1} show the exact solutions and pointwise errors. It can be observed that A-TransNet well captures the anisotropic features of the solution, enabling high-precision numerical solutions for anisotropic problems. The iterative error curves for the three test cases are presented in Figure \ref{fig:5.2.2}, demonstrating the stability of A-TransNet for this class of problems. Table \ref{table5.2} summarizes the comparison results among different methods. Across all three test cases, A-TransNet achieves an accuracy improvement of at least two orders of magnitude under close computational overhead, demonstrating the effectiveness and superiority of A-TransNet for solving PDEs with anisotropic solutions.

\begin{figure}[!htb]
	\centering
	\vspace{-0.2cm}
	\includegraphics[width=0.98\textwidth]{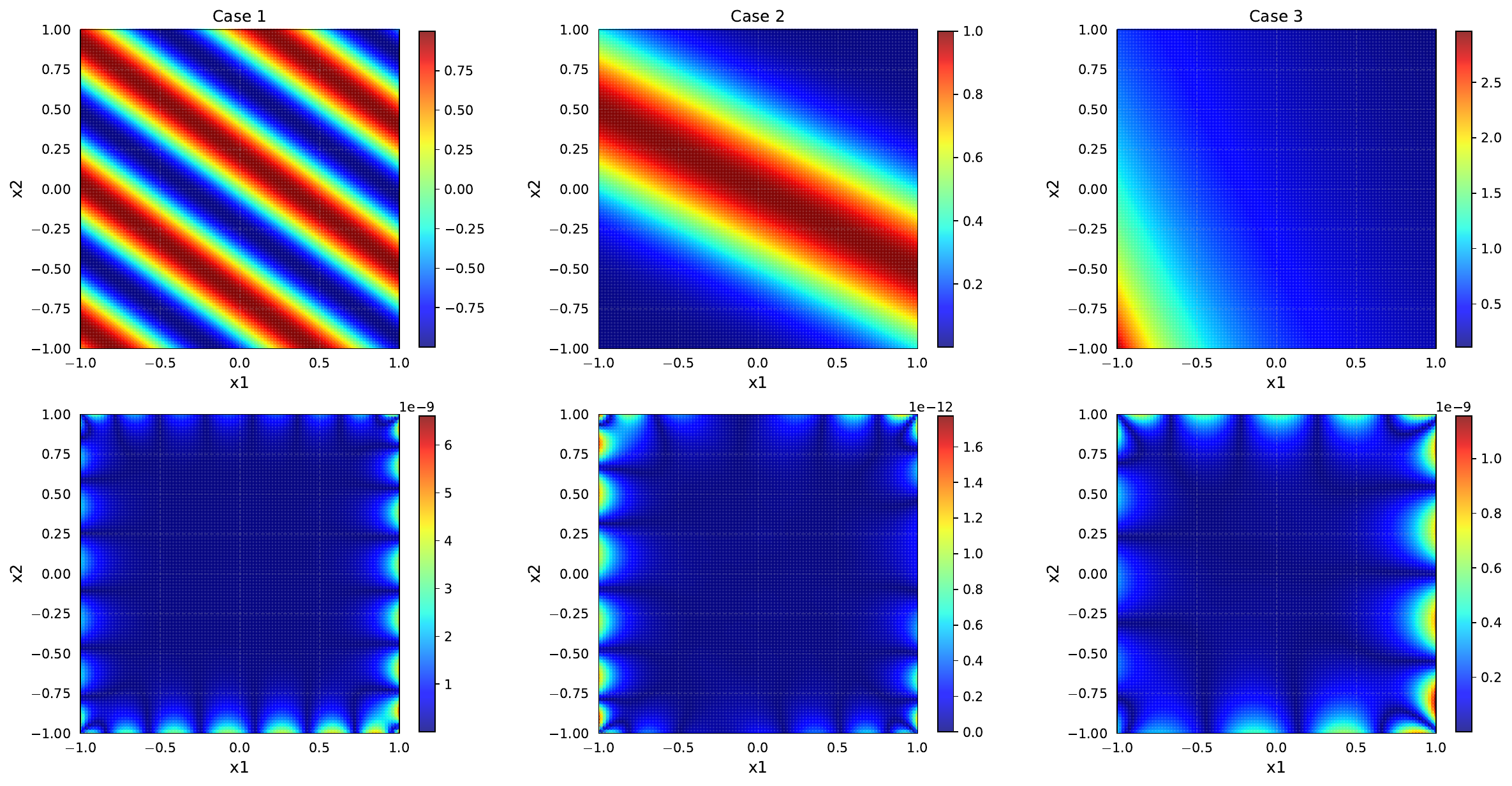}
	\vspace{-0.45cm}
	\caption{Numerical results of A-TransNet for examples in 5.2. Top row: numerical solution; bottom row: absolute error.}
	\label{fig:5.2.1}
\end{figure}

\begin{figure}[!htb]
	\centering
	\vspace{-0.2cm}
	\includegraphics[width=0.98\textwidth]{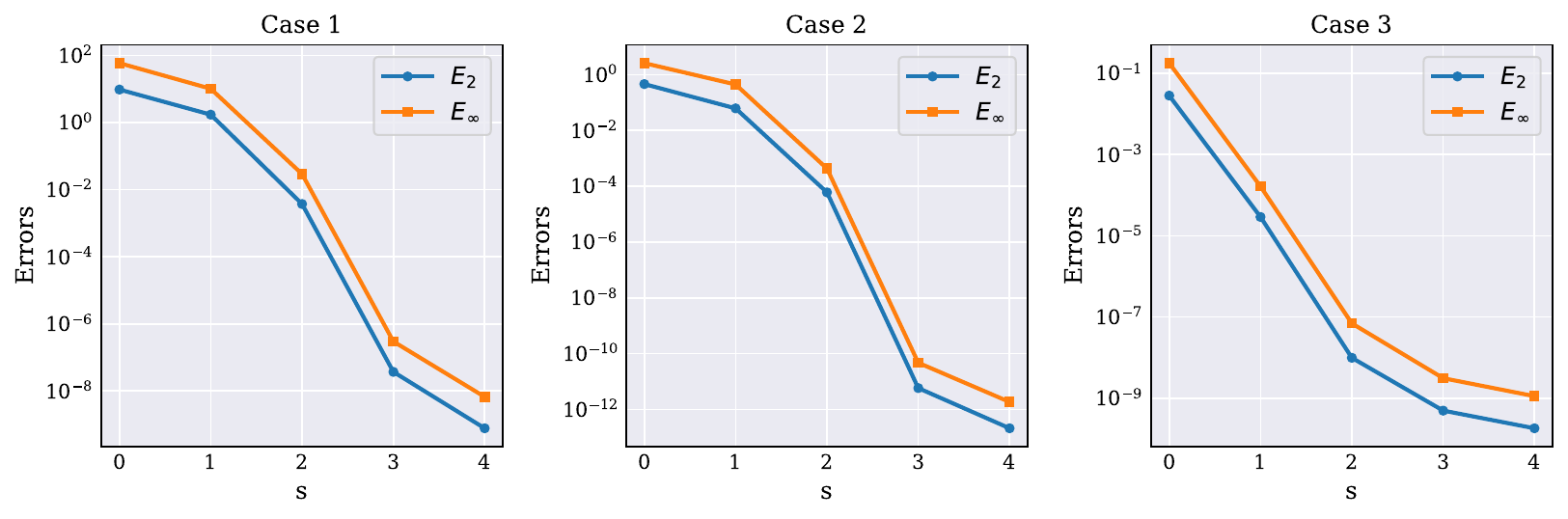}
	\vspace{-0.55cm}
	\caption{Error at each iteration of A-TransNet for examples in 5.2.}
	\label{fig:5.2.2}
	\vspace{-0.1cm}
\end{figure}

\vspace{-0.2cm}
\begin{table}[!htb]
	\centering
	\vspace{-0.1cm}
	\begin{minipage}{\textwidth}
		\centering
		\caption{Comparison of errors for examples in 5.2.}
		\resizebox{\linewidth}{!}{
			\begin{tabular}{cccccccccc}
				\hline
				\multirow{3}{*}{Methods} & \multicolumn{3}{c}{Case 1}& \multicolumn{3}{c}{Case 2} & \multicolumn{3}{c}{Case 3} \\ \cmidrule(lr){2-4}\cmidrule(lr){5-7}\cmidrule(lr){8-10}
				& \multicolumn{3}{c}{$M=800$} & \multicolumn{3}{c}{$M=800$} & \multicolumn{3}{c}{$M=800$} \\ \cmidrule(lr){2-4}\cmidrule(lr){5-7}\cmidrule(lr){8-10}
				& $E_2$ & $E_\infty$ & Time(s) & $E_2$ & $E_\infty$ & Time(s) & $E_2$ & $E_\infty$ & Times(s) \\ \hline
				TransNet & 3.46e-07 & 5.35e-06 & 0.38 & 4.87e-09 & 5.91e-08 & 0.41 & 2.01e-08 & 2.95e-07 & 0.40 \\ 
				PINN & 2.27e-03 & 6.88e-03 & 175.74 & 1.02e-04 & 5.99e-04 & 149.33 & 6.41e-05 & 3.08e-04 & 147.94 \\
				AG-RaNN (SHL) & 4.39e-08 & 5.31e-07 & 0.52 & 2.32e-08 & 2.59e-07 & 0.62 & 1.46e-08 & 2.15e-07 & 0.49 \\ 
				\textbf{A-TransNet} & \textcolor{red}{7.73e-10} & \textcolor{red}{6.68e-09} & \textcolor{red}{0.58} & \textcolor{red}{2.10e-13} & \textcolor{red}{1.85e-12} & \textcolor{red}{0.68} & \textcolor{red}{1.82e-10} & \textcolor{red}{1.12e-09} & \textcolor{red}{0.67} \\\hline
			\end{tabular}
		}
		\label{table5.2}
	\end{minipage}
	\vspace{-0.3cm}
\end{table}

\subsection{Poisson Equation with Rapidly Varying Source Term} \label{Subsec:5.3}

We investigate the Poisson equation \eqref{poisson} on the domain $\Omega = (0, 1)^2$ with the exact solution \cite{Dang2024AGRaNN, Weng2026DCM}:
\begin{align*}
	u(\bx) = 16(1-x_1)x_1(1-x_2)x_2\left( \frac{1}{2} + \frac{1}{\pi}\arctan\left( \frac{r - \|\bx - \boldsymbol{c}\|^2}{\varepsilon} \right) \right).
\end{align*}
The parameter $\varepsilon$ dictates the steepness of the transition region in the solution. Specifically, a smaller $\varepsilon$ value gives rise to steeper gradients and induces more localized variations in the vicinity of the spherical boundary centered at $\boldsymbol{c}$ with radius $\sqrt{r}$. In our numerical experiments, we fix the parameters as $r = 1/16$, $\boldsymbol{c} = (1/2, 1/2)^T$, and $\varepsilon = 1/60$, with the corresponding source term visualized in Figure \ref {fig:5.3.1} (the left two columns). Such steep and localized variations in the source term render the numerical solution of the partial differential equation a highly challenging task \cite{Dang2024AGRaNN, Weng2026DCM}.

\begin{figure}[!htb]
	\centering
	\vspace{-0.2cm}
	\includegraphics[width=0.98\textwidth]{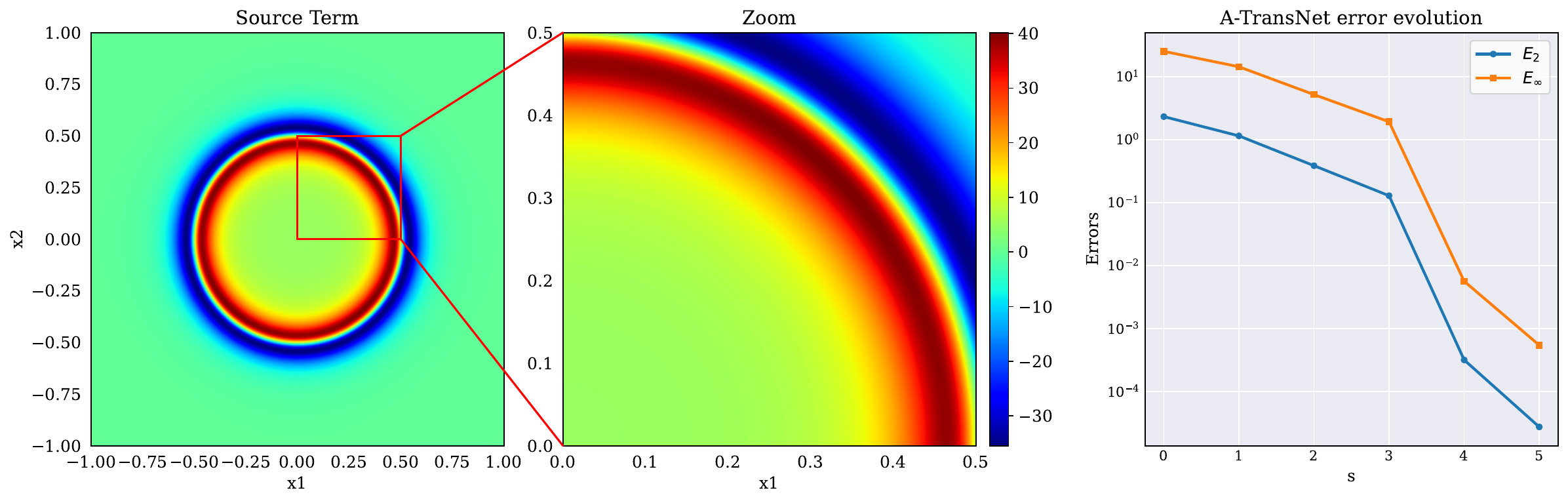}
	\vspace{-0.45cm}
	\caption{The source term corresponds to $\varepsilon = 1/60$ (the left two columns) and error at each iteration of A-TransNet (the right one column) for example in 5.3.}
	\label{fig:5.3.1}
	\vspace{-0.3cm}
\end{figure}

The exact solution, the numerical solution of A-TransNet, and the pointwise absolute error are presented in Figure \ref{fig:5.3.2}. Figure \ref{fig:5.3.1} (the right one column) illustrates the error evolution of A-TransNet during neuron expansion iterations, where the error exhibits a stable decay. Moreover, the superiority of the proposed A-TransNet for such problems can be clearly observed from the numerical results in Table \ref{table5.3}, in comparison with the conventional TransNet. In particular, the TransNet with 16000 neurons achieves an error of $\mathrm{L}_2$ norm of only 6.77e-03 but takes up to 133.28 seconds. In contrast, when iterated to 8000 neurons, A-TransNet obtains an error of 3.19e-04 with a computational time of only 52.68 seconds.

\begin{figure}[!htb]
	\centering
	\vspace{-0.2cm}
	\includegraphics[width=0.98\textwidth]{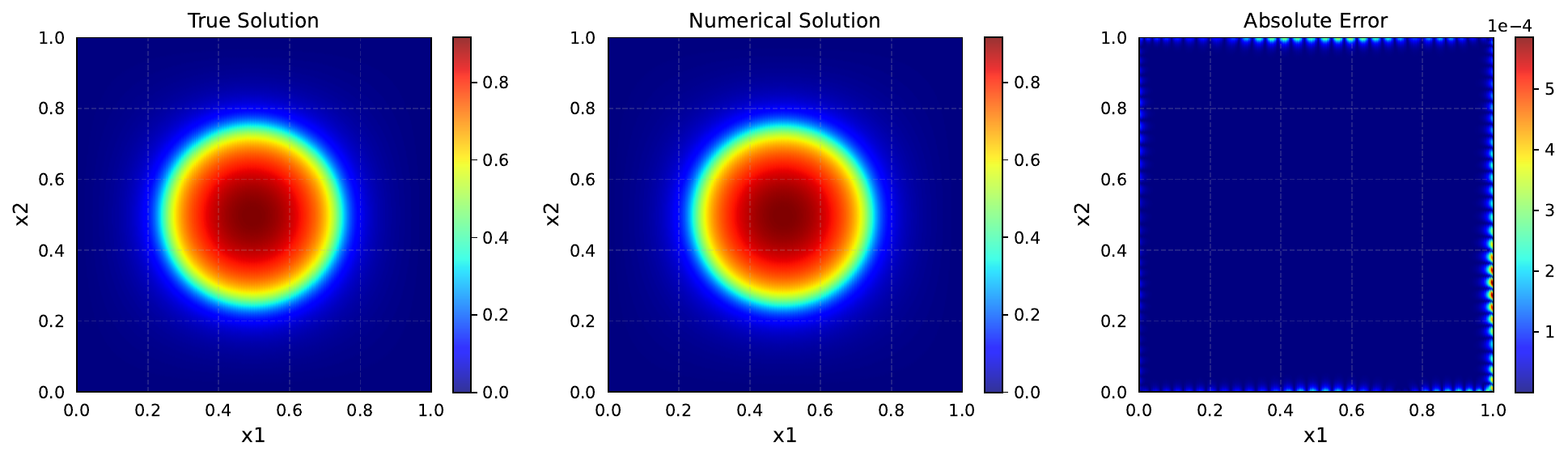}
	\vspace{-0.45cm}
	\caption{Numerical results of A-TransNet for example in 5.3.}
	\label{fig:5.3.2}
	\vspace{-0.25cm}
\end{figure}

\begin{table}[!htbp]
	\centering
	\vspace{-0.4cm}
	\begin{minipage}{0.5\textwidth}
		\centering
		\caption{Comparison of errors for example in 5.3.}
		\resizebox{\linewidth}{!}{
			\begin{tabular}{ccccc}
				\hline
				Methods & M & $E_2$ & $E_\infty$ & Time(s) \\ \hline
				\multirow{2}{*}{TransNet} & 8000 & 2.32e-02 & 4.30e-01 & 31.82 \\
				& 16000 & 6.77e-03 & 1.28e-01 & 133.28 \\
				\multirow{2}{*}{\textbf{A-TransNet}} & 8000 & \textcolor{red}{3.19e-04} & \textcolor{red}{5.65e-03} & \textcolor{red}{52.68} \\ 
				& 16000 & \textcolor{red}{2.76e-05} & \textcolor{red}{5.47e-04} & \textcolor{red}{162.01} \\ \hline
			\end{tabular}
		}
		\label{table5.3}
	\end{minipage}
	\vspace{-0.2cm}
\end{table}

\subsection{Multi-peaked Poisson Equation} \label{Subsec:5.4}

As in \cite{Weng2026DCM}, we consider the Poisson equation \eqref{poisson} on the domain $\Omega = (-1, 1)^2$, which admits a multi-peak exact solution as
\begin{align*}
	u(\bx) = \sum_{k=1}^{K} e^{-250 \| \boldsymbol{x} - \boldsymbol{x}_k \|^2},
\end{align*}
where $K$ denotes the number of peaks, and $\boldsymbol{x}_k$ represents the center of the $k$-th peak. Due to the particularity of the problem, we set the penalty parameters as $\lambda_I=1$ and $\lambda_B=3$ in \eqref{loss function}. We test the following two cases separately.

\begin{description}
	\item[Case 1:] 
	\begin{align*}
		K = 4, \quad
		\boldsymbol{x}_1 = (0.5, 0.5)^T, \quad
		\boldsymbol{x}_2 = (0.5, -0.5)^T, \quad
		\boldsymbol{x}_3 = (-0.5, 0.5)^T, \quad
		\boldsymbol{x}_4 = (-0.5, -0.5)^T.
	\end{align*}
	\item[Case 2:] 
	\begin{align*}
		K = 8, \quad
		&
		\boldsymbol{x}_1 = (0.5, 0.5)^T,\
		\boldsymbol{x}_2 = (0.5, -0.5)^T,\
		\boldsymbol{x}_3 = (-0.5, 0.5)^T,\
		\boldsymbol{x}_4 = (-0.5, -0.5)^T, \\
		&
		\boldsymbol{x}_5 = (0, \sqrt{2}/2)^T,\
		\boldsymbol{x}_6 = (0, -\sqrt{2}/2)^T,\
		\boldsymbol{x}_7 = (\sqrt{2}/2, 0)^T,\
		\boldsymbol{x}_8 = (-\sqrt{2}/2, 0)^T.
	\end{align*}
\end{description}

Figure \ref{fig:5.4.1} display the numerical solutions for the two cases, while the iterative error curves are shown in Figure \ref{fig:5.4.2}. The leftmost column of Figure \ref{fig:5.4.1} visualizes the exact solutions. This problem is challenging due to the multiple sharp peaks in the exact solution, which require the numerical method to capture localized high-gradient features. The corresponding error comparisons with the conventional TransNet are presented in Table \ref{table5.4}. Obviously, A-TransNet performs well on such multi-peaked problems, with an error of $L_2$ norm on the order of $10^{-4}$ at 8000 neurons and $10^{-6}$ at 16000 neurons.

\begin{figure}[!htb]
	\centering
	\vspace{-0.2cm}
	\includegraphics[width=0.98\textwidth]{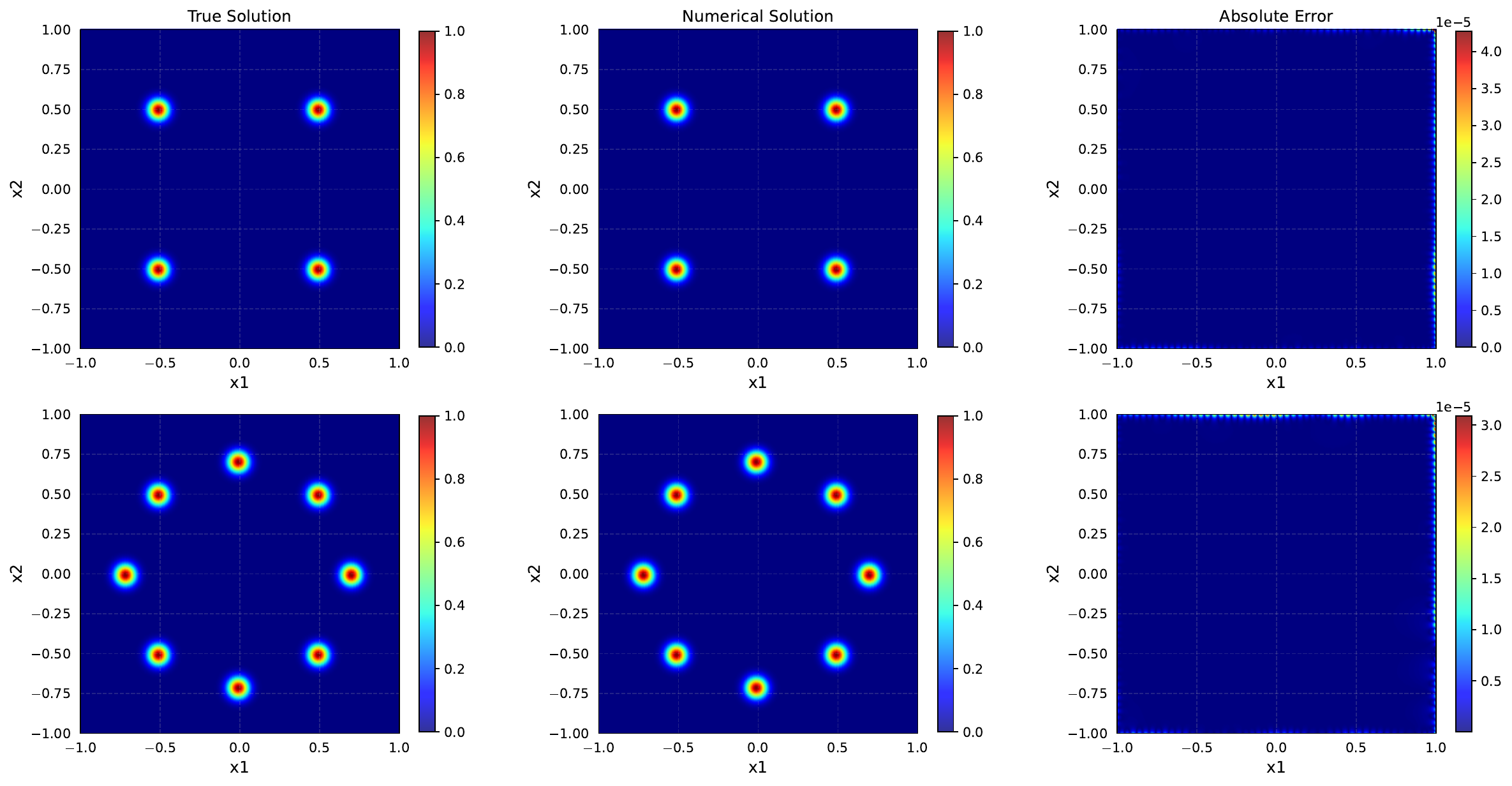}
	\vspace{-0.35cm}
	\caption{Numerical results of A-TransNet for Case 1 (top row) and Case 2 (bottom row) in 5.4.}
	\vspace{-0.3cm}
	\label{fig:5.4.1}
\end{figure}

\begin{figure}[!htb]
	\centering
	\vspace{-0.1cm}
	\includegraphics[width=0.6\textwidth]{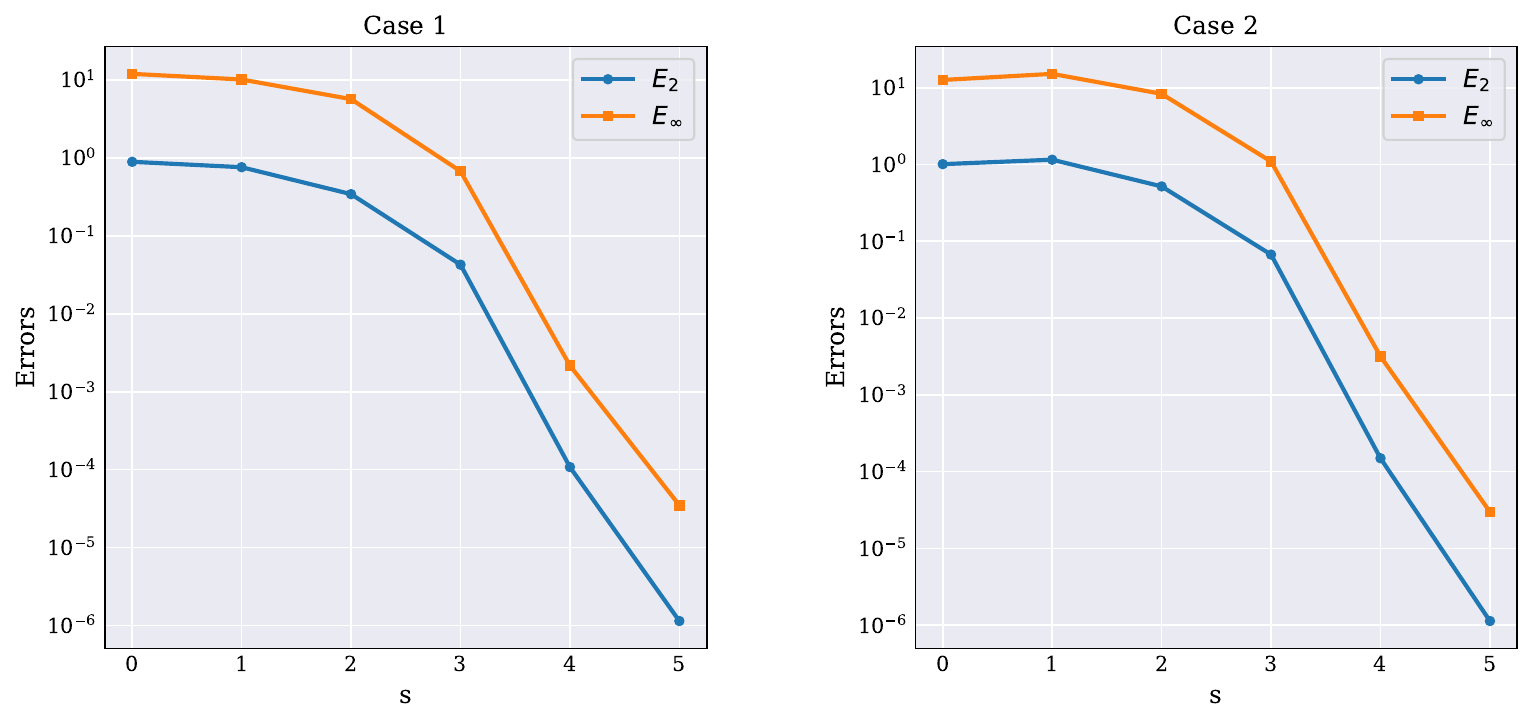}
	\vspace{-0.5cm}
	\caption{Error at each iteration of A-TransNet for examples in 5.4.}
	\label{fig:5.4.2}
	\vspace{-0.15cm}
\end{figure}

\begin{table}[!htb]
	\vspace{-0.2cm}
	\centering
	\begin{minipage}{0.79\textwidth}
		\centering
		\vspace{-0.2cm}
		\caption{Comparison of errors for examples in 5.4.}
		\resizebox{\linewidth}{!}{
			\begin{tabular}{cccccccc}
				\hline
				\multirow{2}{*}{Methods} & \multirow{2}{*}{M} & \multicolumn{3}{c}{Case 1}& \multicolumn{3}{c}{Case 2} \\ \cmidrule(lr){3-5}\cmidrule(lr){6-8}
				& & $E_2$ & $E_\infty$ & Time(s) & $E_2$ & $E_\infty$ & Time(s) \\ \hline
				\multirow{2}{*}{TransNet} & 8000 & 8.42e-04 & 1.88e-02 & 32.12 & 8.28e-04 & 1.47e-02 & 31.87  \\
				& 16000 & 3.01e-05 & 6.93e-04 & 119.35 & 5.36e-05 & 9.31e-04 & 114.85 \\
				\multirow{2}{*}{\textbf{A-TransNet}} & 8000 & \textcolor{red}{1.08e-04} & \textcolor{red}{2.16e-03} & \textcolor{red}{53.26} & \textcolor{red}{1.49e-04} & \textcolor{red}{3.17e-03} & \textcolor{red}{53.52} \\ 
				& 16000 & \textcolor{red}{1.14e-06} & \textcolor{red}{3.49e-05} & \textcolor{red}{155.57} & \textcolor{red}{1.13e-06} & \textcolor{red}{2.99e-05} & \textcolor{red}{156.13} \\ \hline
			\end{tabular}
		}
		\label{table5.4}
	\end{minipage}
	\vspace{-0.2cm}
\end{table}
\vspace{-0.2cm}
\subsection{Biharmonic Equation} \label{Subsec:5.5}

Similar to \cite{Weng2026DCM}, we now test the biharmonic equation given as
\begin{align*}
	\begin{aligned}
		&\Delta^2 u(\boldsymbol{x}) = f(\boldsymbol{x}), \qquad &&\boldsymbol{x} \in \Omega, \\
		&u(\boldsymbol{x}) = \partial_n u(\boldsymbol{x}) = 0, &&\boldsymbol{x} \in \partial \Omega,
	\end{aligned}
\end{align*}
where $\Omega = (-1, 1)^2$ with the exact solution defined as
\begin{align*}
	u(\boldsymbol{x}) = \sin^2(\pi x_1) \sin^2(\pi x_2) + (1 - x_1^2)^4 (1 - x_2^2)^4.
\end{align*}

Figures \ref{fig:5.5.1} display the numerical results and the iterative error evolution curves, respectively, while Table \ref{table5.5} provides a comparison with the results obtained by the conventional TransNet. The results show that A-TransNet still yields a moderate improvement over other methods when dealing with such a relatively smooth multi-order problem, achieving an error of $L_2$ norm of 1.12e-06 with 1200 neurons within 1.04 seconds.

\begin{figure}[!htb]
	\centering
	\vspace{-0.25cm}
	\includegraphics[width=0.98\textwidth]{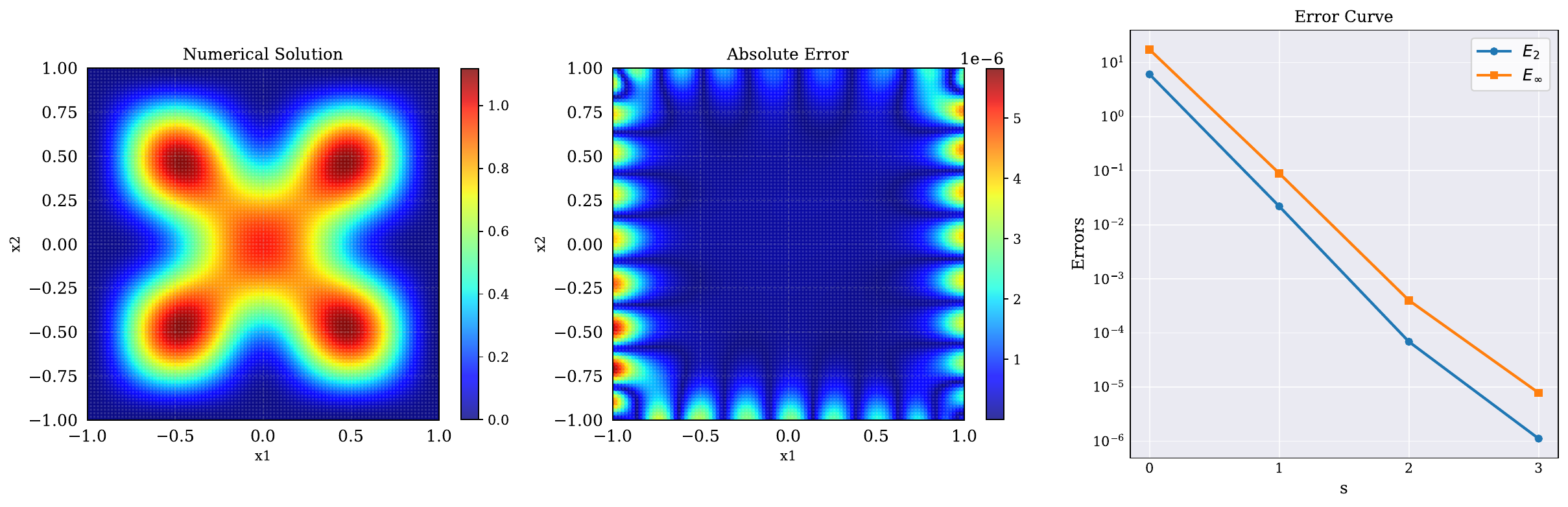}
	\vspace{-0.45cm}
	\caption{Numerical results of A-TransNet for example in 5.5.}
	\label{fig:5.5.1}
	\vspace{-0.1cm}
\end{figure}

\begin{table}[!htbp]
	\centering
	\vspace{-0.3cm}
	\begin{minipage}{0.5\textwidth}
		\centering
		\vspace{-0.2cm}
		\caption{Comparison of errors for example in 5.5.}
		\resizebox{\linewidth}{!}{
			\begin{tabular}{ccccc}
				\hline
				Methods & M & $E_2$ & $E_\infty$ & Time(s)  \\ \hline
				TransNet & 1200 & 9.00e-06 & 6.85e-05 & 0.71 \\
				PINN & 1200 & 5.22e-01 & 1.54e-00 & 794.24 \\
				AG-RaNN (SHL) & 1200 & 7.09e-06 & 4.65e-05 & 0.95 \\ 
				\textbf{A-TransNet} & 1200 & \textcolor{red}{1.12e-06} & \textcolor{red}{7.83e-06} & \textcolor{red}{1.04} \\\hline
			\end{tabular}
		}
		\label{table5.5}
	\end{minipage}
	\vspace{-0.2cm}
\end{table}
\vspace{-0.2cm}
\subsection{Steady-State Heat Conduction Equation with Variable Coefficients} \label{Subsec:5.6}

Consider the steady-state heat conduction equation with variable coefficients given by
\begin{align*}
	\begin{aligned}
		\sum_{i=1}^{d} \frac{\partial}{\partial x_i} \left( k(x_i) \frac{\partial T(\boldsymbol{x})}{\partial x_i} \right) &= f(\boldsymbol{x}), \quad &&\boldsymbol{x} \in \Omega, \\
		T(\boldsymbol{x}) &= 0, \quad &&\boldsymbol{x} \in \partial \Omega,
	\end{aligned}
\end{align*}
where $\Omega = (-1, 1)^d$, $k(x) = 1 - x^2$ is the thermal conductivity \cite{Ling2025Novel, Ling2025Hybrid}. The exact solution to this problem is given by the following two cases.
\vspace{-0.1cm}
\begin{description}
	\item[Case 1:] 
	\begin{align*}
		d=2, \quad T(\bx) = (1 - x_1^2)(1 - x_2^2) \sin\bigl(6\pi (x_1 + \sqrt{2}x_2)\bigr).
	\end{align*}
	\item[Case 2:] 
	\begin{align*}
		d=3, \quad T(\bx) = (1 - x_1^2)(1 - x_2^2)(1 - x_3^2) \sin\bigl(3\pi (x_1 + \sqrt{2}x_2 + 2x_3)\bigr).
	\end{align*}
\end{description}
\vspace{-0.15cm}

Figures \ref{fig:5.6.1} and \ref{fig:5.6.2} present the exact solution, the numerical solution of A-TransNet, and the pointwise error for two cases. Figure \ref{fig:5.6.3} shows the error curves, which decrease stably as the iteration proceeds. We observe that A-TransNet accurately captures the anisotropy inherent in the solution for both two-dimensional and three-dimensional problems, achieving high-precision numerical solutions. Table \ref{table5.6} summarizes the performance comparisons among different methods. Notably, A-TransNet achieves superior performance with high precision: for Case 1, equipped with 4000 neurons, it delivers an accuracy on the order of $10^{-8}$, in contrast to TransNet, which only reaches $10^{-6}$; for Case 2, A-TransNet achieves an error of $L_2$ norm of 1.10e-03 within 73.33 seconds using only 3000 neurons, while other methods can only obtain an error on the order of $10^{-3}$ even with 6000 neurons and more than 100 seconds of computational time.

\begin{figure}[!htb]
	\centering
	\vspace{-0.2cm}
	\includegraphics[width=0.98\textwidth]{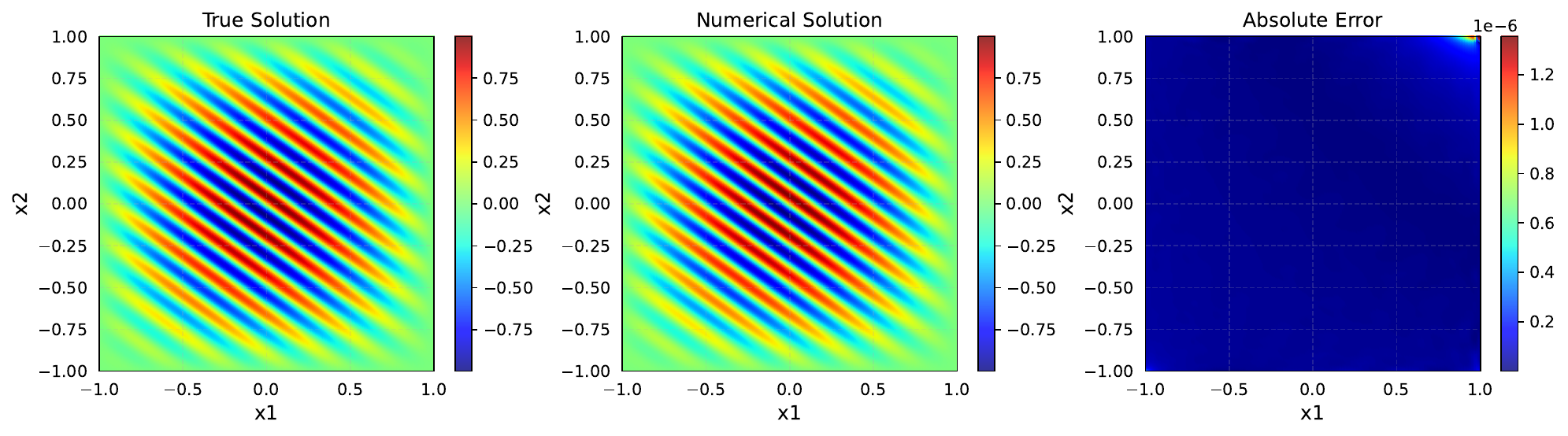}
	\vspace{-0.5cm}
	\caption{Numerical results of A-TransNet for Case 1 in 5.6.}
	\label{fig:5.6.1}
	\vspace{-0.3cm}
\end{figure}

\begin{figure}[!htb]
	\centering
	\vspace{-0.45cm}
	\includegraphics[width=0.98\textwidth]{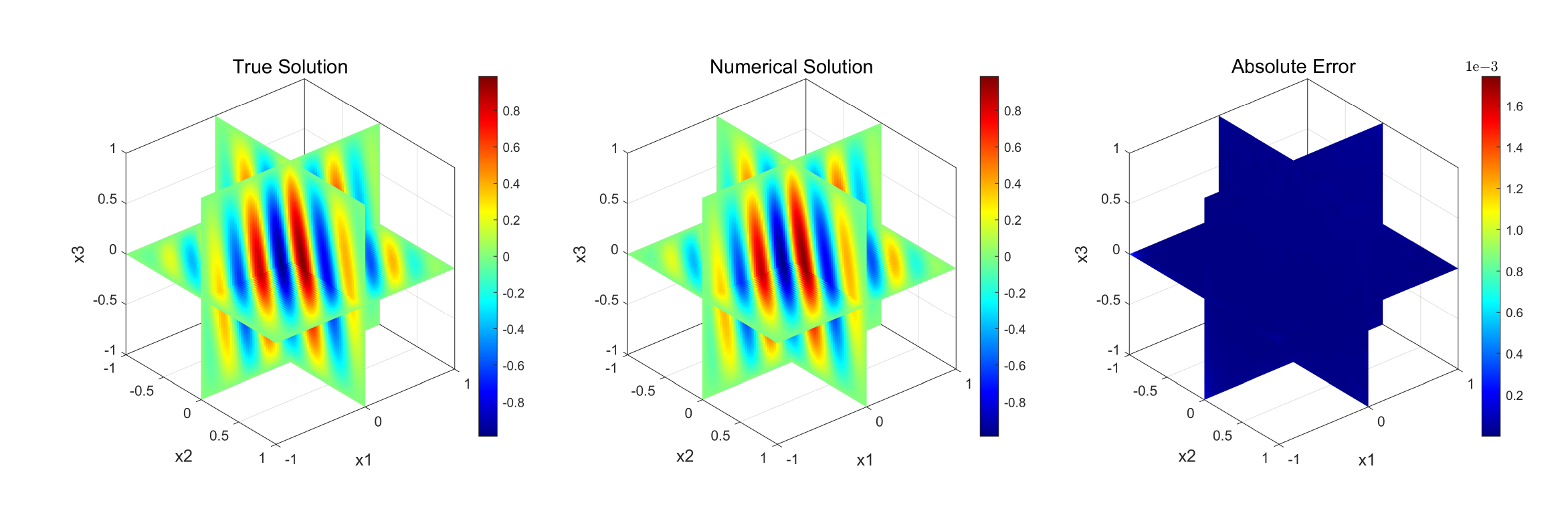}
	\vspace{-0.65cm}
	\caption{Numerical results of A-TransNet for Case 2 in 5.6.}
	\label{fig:5.6.2}
	\vspace{-0.05cm}
\end{figure}

\begin{figure}[!htb]
	\centering
	\vspace{-0.2cm}
	\includegraphics[width=0.57\textwidth]{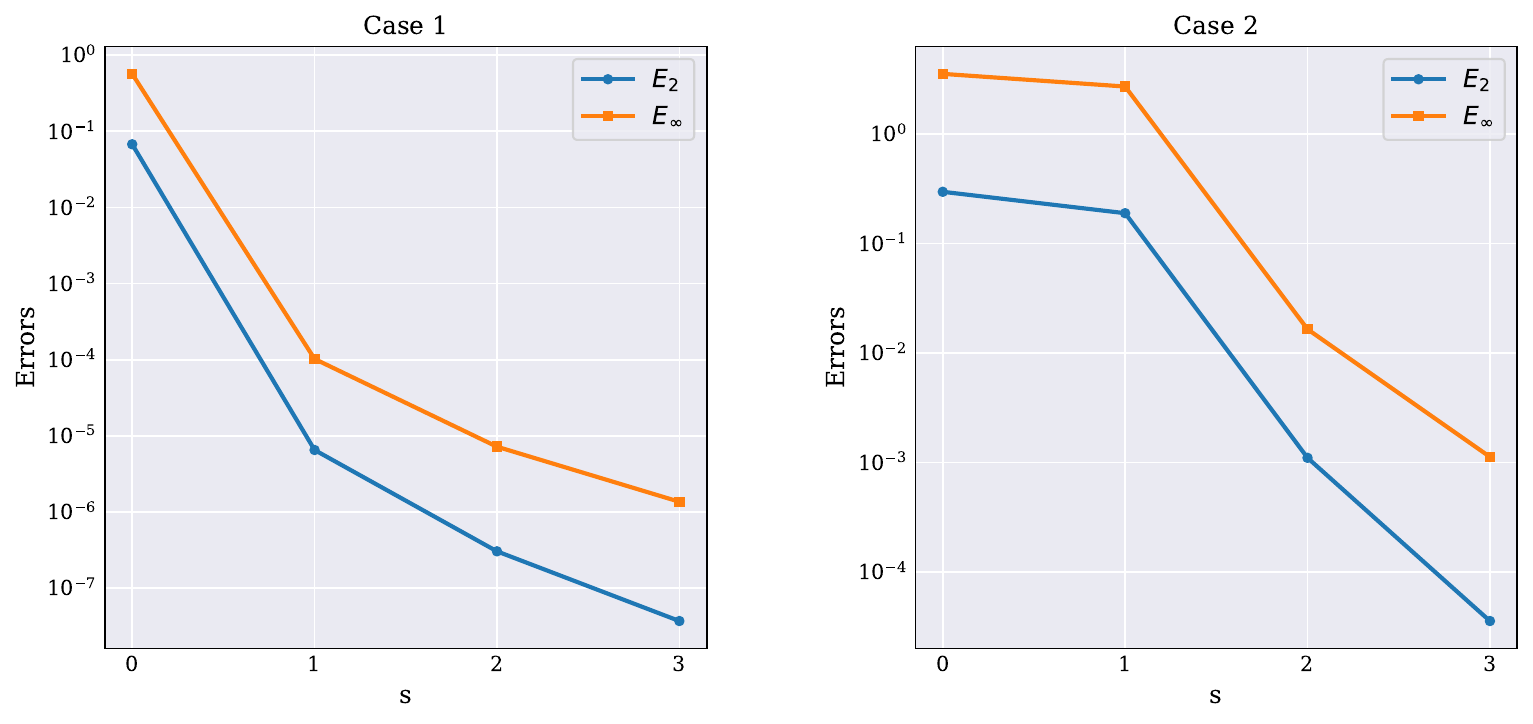}
	\vspace{-0.55cm}
	\caption{Error at each iteration of A-TransNet for examples in 5.6.}
	\label{fig:5.6.3}
\end{figure}

\begin{table}[!htbp]
	\centering
	\begin{minipage}{0.9\textwidth}
		\centering
		\vspace{-0.5cm}
		\caption{Comparison of errors for examples in 5.6.}
		\resizebox{\linewidth}{!}{
			\begin{tabular}{ccccccccc}
				\hline
				\multirow{2}{*}{Methods} &  \multicolumn{4}{c}{Case 1}& \multicolumn{4}{c}{Case 2} \\ \cmidrule(lr){2-5}\cmidrule(lr){6-9}
				& M & $E_2$ & $E_\infty$ & Time(s) & M & $E_2$ & $E_\infty$ & Time(s) \\ \hline
				\multirow{2}{*}{TransNet} & 2000 & 1.41e-04 & 8.31e-04 & 3.10 & 3000 & 1.45e-01 & 1.63e-00 & 41.68 \\
				& 4000 & 1.65e-06 & 3.12e-05 & 7.10 & 6000 & 6.95e-03 & 8.38e-02 & 113.38 \\ 
				\multirow{2}{*}{AG-RaNN (SHL)} & 2000 & 5.34e-05 & 5.49e-04 & 3.98 & 3000 & 2.59e-02 & 3.62e-01 & 58.65 \\
				& 4000 & 3.19e-06 & 2.89e-05 & 8.72 & 6000 & 1.36e-03 & 2.15e-02 & 145.05 \\ 
				\multirow{2}{*}{\textbf{A-TransNet}} & 2000 & \textcolor{red}{3.04e-07} & \textcolor{red}{7.19e-06} & \textcolor{red}{5.69} & 3000 & \textcolor{red}{1.10e-03} & \textcolor{red}{1.65e-02} & \textcolor{red}{73.33} \\ 
				& 4000 & \textcolor{red}{3.70e-08} & \textcolor{red}{1.36e-06} & \textcolor{red}{12.49} & 6000 & \textcolor{red}{3.58e-05} & \textcolor{red}{1.13e-03} & \textcolor{red}{165.62} \\\hline
			\end{tabular}
		}
		\label{table5.6}
	\end{minipage}
	\vspace{-0.35cm}
\end{table}

\subsection{A nonlinear example: the Allen-Cahn equation} \label{Subsec:5.7}

Last, as in \cite{Dang2024AGRaNN, Urban2025PINN}, we focus on the nonlinear Allen–Cahn equation:
\begin{align*}
	\begin{aligned}
		u_t - \varepsilon u_{xx} + \kappa(u^3 - u) &= 0, &(t, x) &\in (0, T) \times (-1, 1), \\
		u(t, -1) &= u(t, 1), &t &\in (0, T), \\
		u_x(t, -1) &= u_x(t, 1), &t &\in (0, T), \\
		u(0, x) &= u_0(x), &x &\in (-1, 1),
	\end{aligned}
\end{align*}
where $\kappa = 5$ , $T = 1$, and $\varepsilon$ is the interface thickness parameter, set to $\varepsilon = 10^{-4}$, yielding an extremely sharp interface and thus a numerically challenging problem for testing the proposed method \cite{Dang2024AGRaNN}. The initial condition is taken as
\begin{align*}
	u_0(x) = x^2 \cos(\pi x).
\end{align*}
Due to the nonlinearity of the equation, Picard iteration is adopted to solve the resulting nonlinear least squares system. Explicit treatment and stabilization of the nonlinear term yield the iterative scheme below:
\begin{align*}
	u^{(k+1)}_t - \varepsilon u^{(k+1)}_{xx} + \kappa S u^{(k+1)}= -\kappa((u^{(k)})^3 - (S+1)u^{(k)}), \quad k=0,1,\dots,K,
\end{align*}
where the stabilizing parameter $S = 1$, the initial value $u^{(0)} = 0$, and the iteration number $K = 20$. Furthermore, since the initial condition need to be strictly satisfied, we set the penalty parameters $\lambda_I=1$ and $\lambda_B=100$ in \eqref{loss function} as in \cite{Urban2025PINN}.

Since the exact solution of this problem is not available, we adopt the high-precision solution proposed in \cite{Urban2025PINN} as the reference solution. Figure \ref{fig:5.7.1} (the left two columns) illustrates the final numerical results of A-TransNet, where the A-TransNet solution and the pointwise error are presented in sequence. The iterative error evolution curves are displayed in Figure \ref{fig:5.7.1} (the right one column). Even for this challenging nonlinear equation whose exact solution has no explicit closed-form expression, the iterative error decays steadily, demonstrating the stability of A-TransNet. Numerical comparisons with TransNet are shown in Table~\ref{table5.7}, which reveals the advantages of A-TransNet. Specifically, A-TransNet achieves an error of $L_2$ norm of 1.27e-04 with only 4000 neurons, compared with 2.96e-02 obtained by TransNet at the same number of neurons.

\begin{figure}[!htb]
	\centering
	\vspace{-0.25cm}
	\includegraphics[width=0.98\textwidth]{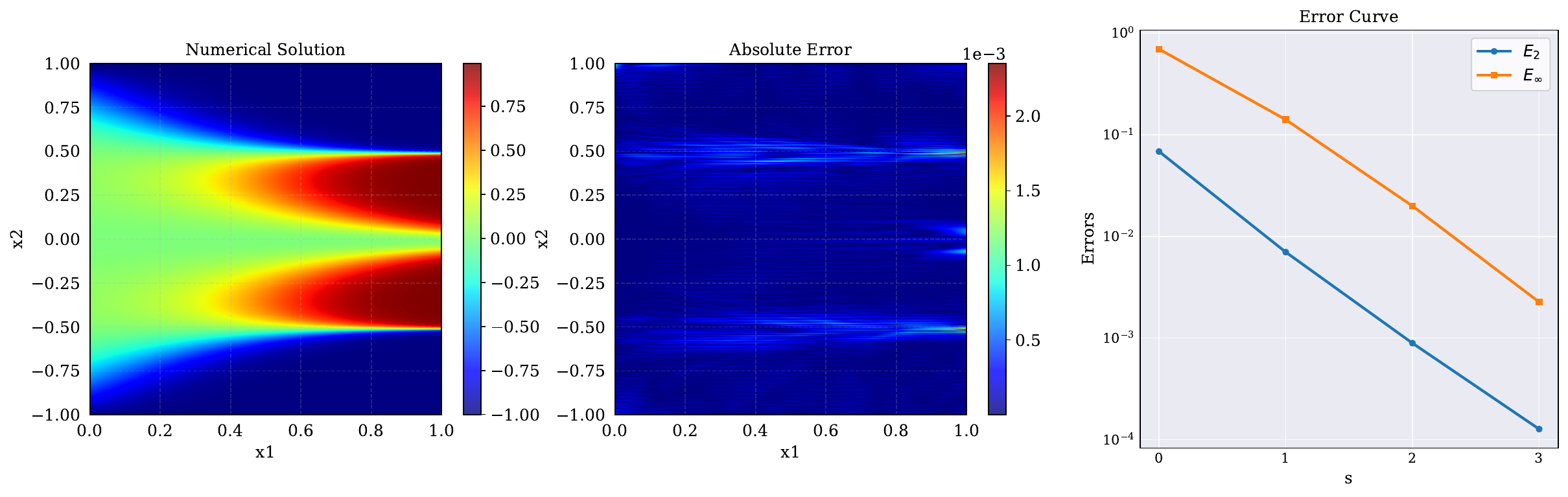}
	\vspace{-0.45cm}
	\caption{Numerical results of A-TransNet for example in 5.7.}
	\label{fig:5.7.1}
	\vspace{-0.4cm}
\end{figure}

\begin{table}[!htbp]
	\centering
	\vspace{-0.5cm}
	\begin{minipage}{0.42\textwidth}
		\centering
		\caption{Comparison of errors for example in 5.7.}
		\resizebox{\linewidth}{!}{
			\begin{tabular}{cccc}
				\hline
				Methods & M & $E_2$ & $E_\infty$ \\ \hline
				\multirow{2}{*}{TransNet} & 2000 & 4.08e-02 & 5.28e-01 \\
				& 4000 & 2.96e-02 & 4.49e-01 \\
				\multirow{2}{*}{\textbf{A-TransNet}} & 2000 & \textcolor{red}{8.90e-04} & \textcolor{red}{1.99e-02} \\ 
				& 4000 & \textcolor{red}{1.27e-04} & \textcolor{red}{2.26e-03} \\ \hline
			\end{tabular}
		}
		\label{table5.7}
	\end{minipage}
	\vspace{-0.3cm}
\end{table}

\subsection{Ablation Study on Allocation Parameter} \label{Subsec:5.8}

To investigate the rationality of the choice of $\mu$ in numerical experiments, a series of ablation studies are carried out. First, for the three test cases in Subsection \ref{Subsec:5.2}, the results of the ablation studies are illustrated in Figure \ref{fig:5.2.3}. It can be found that the influence of $\mu$ on the error remains limited within a reasonable range. Moreover, for problems with anisotropic solutions, a larger value of $\mu$ yields smaller errors, which is consistent with our theoretical analysis.

\begin{figure}[!htb]
	\centering
	\vspace{-0.2cm}
	\includegraphics[width=0.98\textwidth]{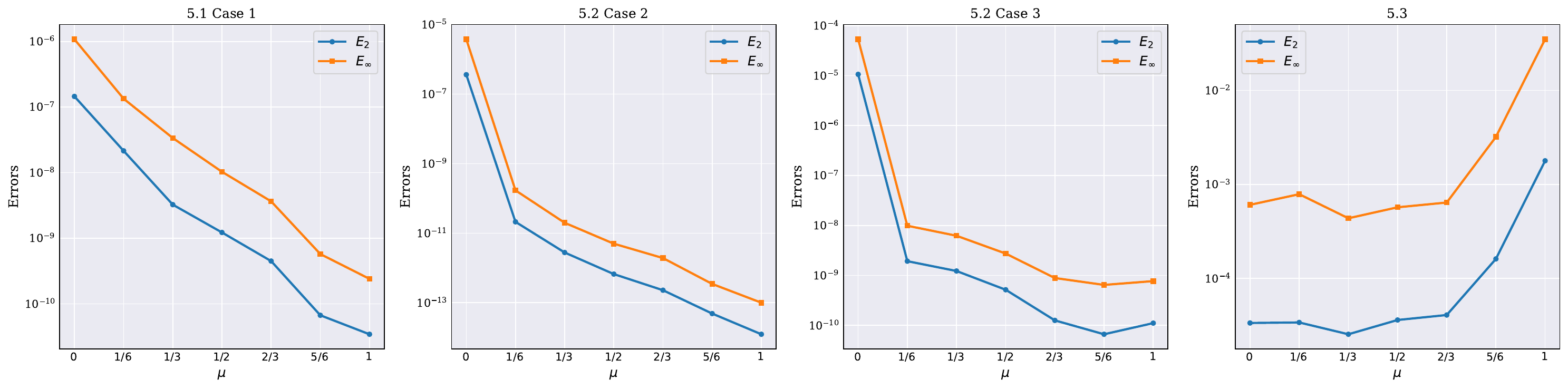}
	\vspace{-0.45cm}
	\caption{Ablation study on the allocation parameter $\mu$ for examples in 5.2 and 5.3.}
	\label{fig:5.2.3}
\end{figure}

Ablation study on $\mu$ is also performed for the example in Subsection \ref{Subsec:5.3}, which is presented in Figure \ref{fig:5.3.2} (right). It can be observed that within a reasonable range, the choice of $\mu$ has a slight influence on the numerical results for this problem, which verifies the rationality of setting $\mu=1/3$.

Figure \ref{fig:5.7.3} presents the ablation results of the parameter $\mu$ for the test case in Subsection \ref{Subsec:5.7}. The left error curves indicate that the model maintains low and stable errors over a reasonable range of $\mu$, demonstrating its insensitivity to this parameter. The right subplots illustrate the absolute error distributions corresponding to $\mu=0$, $\mu=1/3$, and $\mu=1$. In comparison with $\mu=2/3$, $\mu=0$ causes the model to excessively prioritize sharp features, resulting in global imbalance. In contrast, $\mu=1$ renders the model incapable of capturing such features, thereby inducing significant local errors. The results validate that $\mu=2/3$ is a robust choice, which well balances the anisotropy and sharp characteristics of the solution in this problem.

\begin{figure}[!htb]
	\vspace{-0.1cm}
	\centering
	\includegraphics[width=\textwidth]{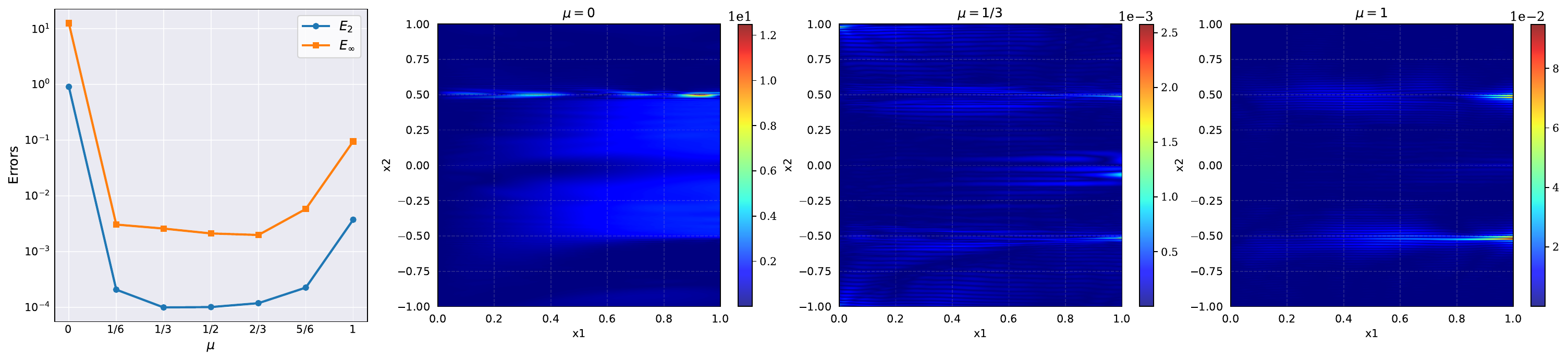}
	\vspace{-0.8cm}
	\caption{Ablation study on the allocation parameter $\mu$ for examples in 5.7.}
	\label{fig:5.7.3}
	\vspace{-0.3cm}
\end{figure}

\subsection{Numerical study of full resampling in iterations} \label{Subsec:5.9}

Another direct iterative strategy is to resample all neurons according to the gradient information of the available approximate solution during iterations. Even though full neuron resampling is unstable in theory, numerical experiments show that the error can still decrease rapidly within a few iterative steps and then become stable under iterations with a fixed neuron count, provided the number of neurons is sufficiently large to ensure a reasonably accurate approximate gradient.

Figure \ref{fig:5.8} presents the error evolution for different test cases: for the example in Subsection \ref{Subsec:5.3}, 8000 neurons are resampled in each iteration; for Case 1 in Subsection \ref{Subsec:5.6}, 2000 neurons are resampled in each iteration; for Case 2 in Subsection \ref{Subsec:5.6}, 3000 neurons are resampled in each iteration; for the example in Subsection \ref{Subsec:5.7}, 2000 neurons are resampled in each iteration. It can be seen that the error exhibits a sharp drop in the first one or two iterations and then stabilizes. This implies that we may also design an iterative strategy with a fixed number of neurons, as proposed in \cite{Deng2025AFCM}. Nevertheless, compared with the iterative neuron expansion strategy, it lacks theoretical guarantees including convergence analysis and a posteriori error estimates, and its maximum achievable accuracy is bounded by the prescribed number of neurons. Even so, this strategy inspires flexible usage of our approach. For instance, resampling with a fixed neuron count may be performed at a certain iteration step within A-TransNet.

\begin{figure}[!htb]
	\centering
	\vspace{-0.2cm}
	\includegraphics[width=\textwidth]{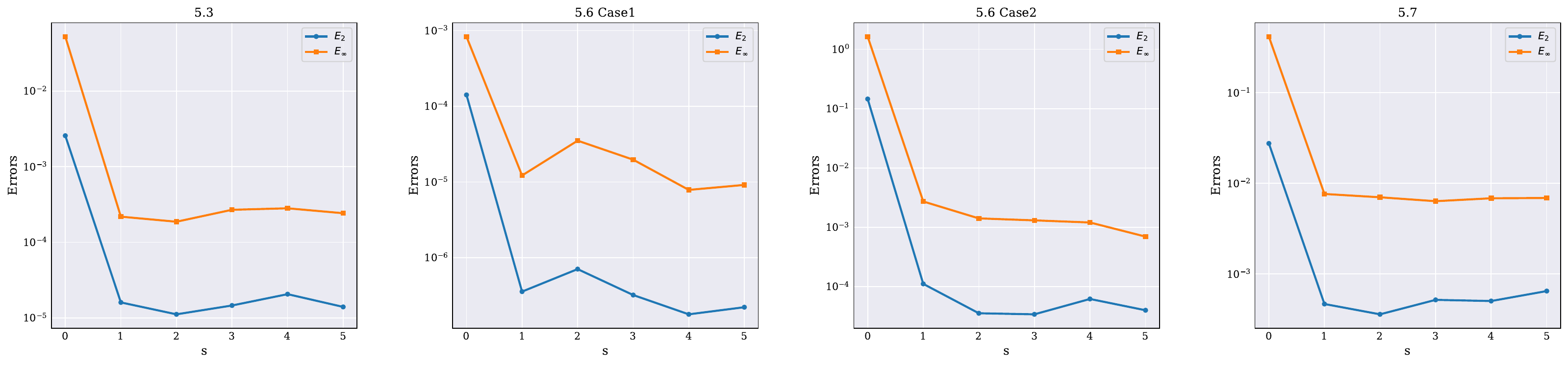}
	\vspace{-0.85cm}
	\caption{Error evolution curves under neuron resampling for different test cases.}
	\label{fig:5.8}
	\vspace{-0.6cm}
\end{figure}

\section{Conclusions} \label{sec:sum}

In this work, we propose a novel machine learning framework, adaptive transferable neural network (A-TransNet), for solving partial differential equations (PDEs) with anisotropic or sharp solution features while achieving high numerical accuracy without additional computational overhead. By leveraging the gradients of the approximate solutions generated by TransNet as adaptation indicators, A-TransNet integrates the active subspaces method and local gradient analysis to adaptively determine the locations of additional neural basis functions. This strategy enables the added feature hyperplanes to align with dominant gradient directions in regions with large solution variations, thereby effectively capturing anisotropic and sharp solution behaviors. Furthermore, the shape parameters of the added neural basis functions are directly determined from the neuron density, allowing flexible adaptation to nonuniform neuron distributions without extra computational costs and enhancing the network representation capability. The adaptive expansion process terminates according to the proposed convergence criterion and a posteriori error estimator, ensuring controllable accuracy improvement. Benefiting from its collocation-based formulation, A-TransNet can be readily applied to a broad range of nonlinear PDEs and problems involving complex geometries. Extensive numerical experiments validate the theoretical analysis and demonstrate that A-TransNet achieves superior accuracy and efficiency compared with state-of-the-art methods for PDEs with anisotropic or sharp solutions. By enabling neuron-level adaptivity, A-TransNet provides an effective and convergent framework for addressing PDEs with complex solution characteristics.Future work will focus on extending this framework to generalized transferable neural networks (GTransNet) in \cite{Cheng2026GTransNet}  and exploring its applicability to broader scientific computing problems, including higher-dimensional Maxwell’s equations and computational fluid dynamics with complex geometries.

\vspace{1em}
\noindent\textbf{CRediT authorship contribution statement}

Jiacong Qu: Writing -- review \& editing, Writing -- original draft, Software, Methodology, Investigation, Visualization. Liyong Zhu: Writing -- review \& editing, Supervision, Methodology, Conceptualization, Formal analysis, Funding acquisition, Resources, Writing -- original draft.

\vspace{1em}
\noindent\textbf{Declaration of competing interest}

The authors declare that they have no known competing financial interests or personal relationships that could have appeared to influence the work reported in this paper.

\vspace{1em}
\noindent\textbf{Acknowledgement}

L. Zhu’s work was partially supported by Science Challenge Project under grant number TZ2026004 and National Natural Science Foundation of China under grant number 12671514.

\appendix
\section{Proof of Theorem~\ref{neuron density consistency}}\label{app:proofthm2}

\begin{proof}[Proof of Theorem~\ref{neuron density consistency}]
	Since the location parameters $\{\boldsymbol{a}_m\}_{m=1}^M$ and $\{r_m\}_{m=1}^M$ are sampled as
	\begin{align*}
		\boldsymbol{a}_m \stackrel{\text{i.i.d.}}{\sim} U(S^{d-1}), \quad
		r_m \stackrel{\text{i.i.d.}}{\sim} U[-R, R],
	\end{align*}
	and undergo the coordinate transformation given in \eqref{coordinate transformation}, with sufficiently large $M$ and sufficiently small $\delta$, their respective probability densities after transformation are
	\begin{align*}
		p(\boldsymbol{\theta}_{m}) = \frac{1}{\pi^{d-1}}, \quad
		p(l_m) = \frac{1}{2R}.
	\end{align*}
	As $\boldsymbol{a}_m$ and $r_m$ are mutually independent, the transformed variables $\boldsymbol{\theta}_m$ and $l_m$ remain independent, leading to the joint probability density
	\begin{align*}
		p(\boldsymbol{\theta}_m, l_m) = p(\boldsymbol{\theta}) \cdot p(l_m) = \frac{1}{2\pi^{d-1} R}.
	\end{align*}
	
	Consider a small rectangular neighborhood $\prod_{i=1}^{d-1} \left[\theta_{mi}-\frac{\pi \delta}{2R}, \theta_{mi} + \frac{\pi \delta}{2R}\right] \times [l_m-\delta, l_m +\delta]$, where the parameter $\delta$ determines the size of the neighborhood, and whose volume is
	\begin{align*}
		\text{Vol}(\delta) = \left(\frac{\pi \delta}{R}\right)^{d-1} \cdot 2\delta = \frac{2\pi^{d-1} \delta^d}{R^{d-1}}.
	\end{align*}
	From this, the probability that a single neuron falls into this neighborhood can be computed as
	\begin{align*}
		\mathbb{P}(\delta) = p(\theta, l) \cdot \text{Vol}(\delta) = \frac{1}{2\pi^{d-1} R} \cdot \frac{2\pi^{d-1} \delta^d}{R^{d-1}} = \frac{\delta^d}{ R^d}.
	\end{align*}
	Applying the linearity of expectation, the expected number of neurons in the neighborhood is given by
	\begin{align*}
		\mathbb{E}[N_m] = M \cdot \mathbb{P}(\delta) = \frac{M \delta^d}{R^d}.
	\end{align*}
	
	Substituting $\mathbb{E}[N_m]$ into the definition of $\rho_m$, we find that
	\begin{align*}
		\mathbb{E}[\rho_m] = \frac{\mathbb{E}[N_m] R^{d-1}}{2\pi^{d-1} \delta^d}
		= \frac{\frac{M \delta^d}{R^d} \cdot R^{d-1}}{2\pi^{d-1} \delta^d}
		= \frac{M}{2\pi^{d-1} R},
	\end{align*}
	which gives the exact expectation of the neuron density. For large $M$ and small $\delta$, this implies the asymptotic consistency of the sample neuron density.
	
	Lastly, taking the expectation of the generation formula \eqref{shape parameters formula} gives
	\begin{align*}
		\mathbb{E}[\gamma_m] \approx C \frac{\mathbb{E}[N_m]^{1/d}}{\delta} = C \frac{M^{1/d}}{R}.
	\end{align*}
	This completes the proof.
\end{proof}

\vspace{1em}
\noindent\textbf{Data availability}

Data will be made available on request.

\end{document}